\documentclass[reqno,12pt,a4paper]{amsart}
\usepackage[T1]{fontenc}
\usepackage[utf8]{inputenc}
\usepackage[english]{babel}
\usepackage{amsmath,amssymb,amsthm}
\usepackage{bm}
\usepackage[margin=25mm]{geometry}
\usepackage[hidelinks]{hyperref}
\newcommand{\cN}{\mathbb N}
\newcommand{\cZ}{\mathbb Z}
\newcommand{\PP}{\mathbb P}
\newcommand{\cA}{\mathcal A}
\newcommand{\cB}{\mathcal B}

\theoremstyle{plain}
\newtheorem{theorem}{Theorem}[section]
\newtheorem{lemma}{Lemma}[section]

\theoremstyle{definition}

\theoremstyle{remark}
\newtheorem{remark}{Remark}
\numberwithin{equation}{section}

\newcommand{\cR}{\mathcal R}
\newcommand{\E}{\mathbb E}
\newcommand{\ind}{\mathbf 1}
\allowdisplaybreaks[2]

\title[Additive complements and Romanoff gaps]
{Additive complements to the primes and large gaps between Romanoff numbers}
\author{Artyom Radomskii}
\address{HSE University, Moscow, Russian Federation}
\keywords{Prime numbers, additive complements, Romanoff numbers,
large gaps, sieve methods}
\subjclass[2020]{Primary 11P32; Secondary 11B05, 11B13, 11N36}

\email{artyom.radomskii@mail.ru}
\date{}

\begin{document}
\begin{abstract}
We prove that every set $\mathcal{B}\subseteq\mathbb{Z}_{\ge0}$ for which
$\mathbb{P}+\mathcal{B}$ contains all sufficiently large integers satisfies
$B(X)\gg\log X\log\log X$, where
$B(X)=\#(\mathcal{B}\cap[0,X])$ and the implied constant is absolute.
This proves a quantitative form of Ruzsa's Conjecture~2 and gives a
negative answer to the $O(\log X)$ question in Erd\H{o}s Problem~32.
The bound follows from a more general result: for an arbitrary set
$\mathcal{B}\subseteq\mathbb{Z}_{\ge0}$ and sufficiently large $X$, we
construct a block of
$\lfloor c\log X\log\log X/(B(X)+1)\rfloor$ consecutive integers
in $[X/2,X]$ outside $\mathbb{P}+\mathcal{B}$, provided this length is positive.
As an application, we obtain $G_{\mathcal R}(X)\gg\log\log X$ for the
longest block in $[1,X]$ containing no Romanoff number $p+2^n$,
improving the bound
$G_{\mathcal R}(X)\gg\log\log X/\log\log\log X$ of Kalmynin and Konyagin.
The argument also applies to $p+a^n$ for every fixed integer $a\ge2$.
The key step is a lemma on large shifts, proved using truncated products
of divisor sums and the removal of auxiliary primes dividing differences
of shifts.
\end{abstract}
\maketitle

\section{Introduction}\label{sec:introduction}

Let $\PP$ denote the set of primes and let $\cN=\{1,2,\ldots\}$.
For a set $\cB\subseteq\cZ_{\ge0}$, put
\[
 B(X)=\#(\cB\cap[0,X]),\qquad
 \PP+\cB=\{p+a:p\in\PP,\ a\in\cB\}.
\]
We call $\cB$ an \emph{asymptotic additive complement to the primes}
if there is an integer $n_0\ge1$ such that every integer $n\ge n_0$
belongs to $\PP+\cB$. We are interested in lower bounds for $B(X)$.

Erd\H{o}s constructed such a set with $B(X)=O((\log X)^2)$;
see~\cite{ErdosProblem32,Ruzsa} and the
references therein. Problem~32 in the Erd\H{o}s Problems
collection~\cite{ErdosProblem32} asks whether this can be improved
to $B(X)=o((\log X)^2)$, or even to $B(X)=O(\log X)$.
A further question is whether every such complement satisfies
\[
 \liminf_{X\to\infty}\frac{B(X)}{\log X}>1.
\]
Ruzsa~\cite[Theorem~2]{Ruzsa} answered this last question by proving
the stronger estimate
\[
 \liminf_{X\to\infty}\frac{B(X)}{\log X}\ge e^{\gamma_E},
\]
where $\gamma_E$ is Euler's constant.

Ruzsa further conjectured~\cite[Conjecture~2, p.~270]{Ruzsa} that
every asymptotic additive complement satisfies
\begin{equation}\label{eq:ruzsa}
 \frac{B(X)}{\log X}\longrightarrow\infty
 \qquad(X\to\infty).
\end{equation}
We prove a quantitative form of \eqref{eq:ruzsa} using
Lemma~\ref{lem:translate}. The main result gives a block of
consecutive integers outside $\PP+\cB$ for an arbitrary set $\cB$.
\begin{theorem}[A block outside a prime sumset]\label{thm:block}
There exist absolute constants $c_*>0$ and $X_*>0$ with the following
property. For every set $\cB\subseteq\cZ_{\ge0}$ and every real
$X\ge X_*$, put
\begin{equation}\label{eq:H}
 H=\left\lfloor c_*\frac{\log X\log\log X}{B(X)+1}\right\rfloor.
\end{equation}
If $H\ge1$, there is an integer $N$ such that
\[
 \{N+1,\ldots,N+H\}\subseteq[X/2,X]
\]
and
\[
 \{N+1,\ldots,N+H\}\cap(\PP+\cB)=\varnothing.
\]
Both constants are independent of $\cB$.
\end{theorem}

\begin{theorem}[Asymptotic additive complements]\label{cor:complement}
There is an absolute constant $c>0$ such that every asymptotic
additive complement $\cB\subseteq\cZ_{\ge0}$ to the primes satisfies
\begin{equation}\label{eq:lower}
 B(X)\ge c\log X\log\log X
 \qquad(X\ge X_0(\cB)).
\end{equation}
In particular, \eqref{eq:ruzsa} holds.
\end{theorem}

Theorem~\ref{cor:complement} rules out $B(X)=O(\log X)$,
answering this part of Problem~32 in~\cite{ErdosProblem32}.
The estimate \eqref{eq:lower} leaves open the question whether
$B(X)=o((\log X)^2)$ is possible. The assumption that $\PP+\cB$
contains every sufficiently large integer is essential here; the
argument does not establish Ruzsa's Conjecture~1 in~\cite{Ruzsa},
where $\PP+\cB$ is only assumed to have asymptotic density~$1$.

Theorem~\ref{thm:block} also has a natural application to gaps between Romanoff numbers.
We therefore consider the set
\[
 \cR=\{p+2^n:p\in\PP,\ n\in\cN\}.
\]
For $X\ge1$, define
\[
 \begin{aligned}
 G_{\cR}(X)=\max\bigl(&\{0\}\cup
 \{H\in\cN:\ \exists N\in\cZ_{\ge0},\ N+H\le X,\\
 &\hspace{33mm}\{N+1,\ldots,N+H\}\cap\cR=\varnothing\}\bigr).
 \end{aligned}
\]

The study of representations involving
primes and powers of two goes back to Euler and Goldbach; see Euler's
letter to Goldbach dated 16 December 1752~\cite{EulerGoldbach}.
De Polignac's paper of 1849~\cite[Th\'eor\`eme~2, IV]{Polignac} is
another early contribution to this subject. In 1934, Romanoff~\cite{Romanoff}
proved that $\cR$ has positive lower density. Romanoff's theorem has since been generalized in several directions. Rieger~\cite{Rieger} extended it to algebraic number
fields. Shparlinski and Weingartner~\cite{ShparlinskiWeingartner}
proved an explicit analogue over finite fields, estimating the
proportion of polynomials representable as $h+g^k$, where $h$ is
irreducible and $g$ is a fixed polynomial of positive degree. Radomskii~\cite[Theorem~1.1]{Radomskii} established a general criterion
for many integers to have many representations $a_i+b_j$.
It combines counting and pair-correlation estimates for
$\cA=\{a_i\}$ with an averaged congruence condition for
$\cB=\{b_j\}$, explicitly allowing repeated terms in $\cB$.
Applications include prime summands and sums of two squares,
combined with polynomial values, restricted powers, and
elliptic-curve point counts.

The complementary problem concerns integers outside $\cR$.
Erd\H{o}s~\cite{Erdos} constructed, using covering congruences,
an arithmetic progression of odd integers containing no Romanoff
numbers. In particular, the upper asymptotic density of $\cR$ is
strictly less than $1/2$. Habsieger and Roblot~\cite[\S3]{HabsiegerRoblot}
subsequently obtained the explicit bound
\[
 \limsup_{X\to\infty}\frac{\#(\cR\cap[1,X])}{X}<0.490941
\]
by computations with residue classes.

Kalmynin and Konyagin~\cite[Theorem~1]{KK} proved that
\[
 G_{\cR}(X)\gg\frac{\log\log X}{\log\log\log X}.
\]

As a consequence of Theorem \ref{thm:block}, we obtain the following improvement.

\begin{samepage}
\begin{theorem}\label{cor:base}
Let $a\ge2$ be a fixed integer and let
$\cR_a=\{p+a^n:p\in\PP,\ n\in\cN\}$. Then
\[
 G_{\cR_a}(X)\gg_a\log\log X
\]
for all sufficiently large $X$, where $G_{\cR_a}$ is defined
analogously to $G_{\cR}$.
\end{theorem}
\end{samepage}

The proof of Theorem~\ref{thm:block} proceeds by covering the distinct shifts $j-a$, with
$1\le j\le H$ and $a\in\cB\cap[0,X]$, by congruences modulo primes.
Write $x=\frac13\log X$. Averaging over residue classes modulo the
product of the primes up to $x$ leaves $O(HB(X)/\log x)$ shifts
to be covered. By \eqref{eq:H}, their number is at most a sufficiently
small constant times $x$, so the large-shift lemma applies.

To ensure that the prime divisors obtained in this way are proper,
we place the block in the middle third of an interval disjoint from
$\cB$. For each $a\in\cB$, the difference between an integer in the
block and $a$ is then either negative or larger than all the primes
used in the construction.

We use the weight construction of \cite{LG}, with the modifications needed to handle large shifts. The product cutoff in this construction is analogous to the simplex restriction used in the Maynard--Tao
sieve; see Maynard~\cite[\S~7]{Maynard}. Proposition~1.2 of~\cite{LG} is stated for shifts of small diameter.
Our problems involve large negative shifts $j-a$, so a separate
adaptation of its proof is required. This is given in detail in
Section~\ref{sec:lemma}. 

After fixing notation in Section~\ref{sec:notation}, we prove
the large-shift lemma in Section~\ref{sec:lemma}.
Theorem~\ref{thm:block} is proved in Section~\ref{sec:block-proof}.
In Section~\ref{sec:complements}, we deduce
Theorems~\ref{cor:complement} and~\ref{cor:base} from
Theorem~\ref{thm:block}.

\section{Notation}\label{sec:notation}

We use $(m,n)$ for the greatest common divisor of integers $m,n$,
and $\sum_{r\bmod q}$ for a sum over a complete system of residue
classes modulo $q$. Unless stated otherwise, sums indexed by $p$ or $q$
are over primes satisfying the indicated conditions. The usual floor
and ceiling functions are denoted
by $\lfloor v\rfloor$ and $\lceil v\rceil$. Empty sums and products are
understood to be $0$ and $1$, respectively. The symbol
$\ind_{\mathcal E}$ denotes the indicator of a condition $\mathcal E$.
For $n\ge1$, we write $\varphi(n)$ for Euler's totient function and put
\[
 \omega(n)=\#\{p\in\PP:p\mid n\},\qquad
 \sigma_1(n)=\sum_{d\mid n}d.
\]
All divisors in sums and products are positive unless stated otherwise.
Euler's constant is denoted by $\gamma_E$.

For $g>0$, we write $f\ll g$ or $f=O(g)$ if $|f|\le Cg$
for a positive constant $C$; as usual, $f\gg g$ and $f\asymp g$
have the corresponding meanings, while $f=o(g)$ means $f/g\to0$.
Implied constants are absolute unless a dependence is indicated,
as in $\ll_a$. Limits are taken as $x$ or $X$ tends to infinity,
as appropriate. The constants $C_1\ge1$ and $\eta>0$, once chosen,
remain fixed in the proof.

\section{A lemma on large shifts}\label{sec:lemma}

\begin{lemma}\label{lem:translate}
There exists an absolute constant $\eta>0$ with the following property.
Let $x$ be sufficiently large, let $S=\{s_1,\ldots,s_k\}$ be
nonempty, and suppose that
\[
 Q=\prod_{p\le x}p,
 \qquad s_1,\ldots,s_k\in\cZ,
 \qquad |s_i|\le e^{3x},
 \qquad k\le\eta x,
\]
where the shifts $s_i$ are pairwise distinct. Then, for every
$b\in\cZ$, there exists an integer $t$ with $1\le t\le e^x$ such
that each number $b+Qt+s_i$ is divisible by a prime in the interval
\[
 \left(x^6,\ \exp\!\left(\frac{x}{(\log x)^5}\right)\right].
\]
\end{lemma}

\begin{remark}
The lemma asserts divisibility only. Positivity of the values and the
fact that the resulting prime divisors are proper are checked when the
lemma is applied. 
\end{remark}

The proof of Lemma~\ref{lem:translate} occupies the rest of this
section. We first deal with the primes at which two forms have a
common root, then construct a truncated weight. The final two
estimates compare its total mass with the contribution of uncovered
forms.

We may assume that $\eta\le1$. Fix the parameters
\begin{equation}\label{eq:parameters}
 \begin{gathered}
 z=x^6,\qquad y=\exp\!\left(\frac{x}{(\log x)^5}\right),
 \qquad D=e^{x/4},\qquad T=\lfloor e^x\rfloor,\\
 \beta=\frac1{\log y},\qquad
 u=\frac{\log D}{\log y}=\frac14(\log x)^5.
 \end{gathered}
\end{equation}

\subsection{Exceptional primes and divisor sums}

Put
\[
 \Delta=\prod_{1\le i<j\le k}|s_i-s_j|,
\]
where an empty product is understood to be $1$.
Define
\[
 P_0=\prod_{z<p\le y}p,\qquad
 P_* =\prod_{\substack{z<p\le y\\p\nmid\Delta}}p.
\]
For $1\le i\le k$, write $L_i(t)=b+Qt+s_i$.

\begin{lemma}[Exceptional primes]\label{lem:exceptional-primes}
We have
\begin{equation}\label{eq:exceptional}
 \sum_{\substack{p>z\\p\mid\Delta}}\frac1p
 \le\frac{\log\Delta}{z\log z}
 \ll\frac1{x^3\log x}.
\end{equation}
Moreover, for every $p\mid P_*$, the congruences
$L_i(t)\equiv0\pmod p$, $1\le i\le k$, have distinct roots.
\end{lemma}

\begin{proof}
Since $k\le\eta x$ and $|s_i|\le e^{3x}$, we have
\[
 \log\Delta\le\binom{k}{2}(3x+\log2)=O(x^3).
\]
We have
\[
\#\{p: p>z, p|\Delta\}\le \sum_{\substack{p>z\\ p|\Delta}}\frac{\log p}{\log p}
\le \frac{1}{\log z}\sum_{p|\Delta}\log p\le \frac{\log \Delta}{\log z}.
\]
This proves \eqref{eq:exceptional}.
For every $p\mid P_*$, we have $p>z>x$, and hence
$(p,Q)=1$. Therefore, $Q$ is invertible modulo $p$,
and each congruence
\[
L_i(t)=b+Qt+s_i\equiv0\pmod p
\]
has exactly one root, namely
\[
t\equiv -(b+s_i)Q^{-1}\pmod p.
\]
We now show that these roots are pairwise distinct.
Indeed, if the roots corresponding to two indices
$i\ne j$ coincided, then there would exist
an integer $t$ such that
\[
b+Qt+s_i\equiv b+Qt+s_j\equiv0\pmod p.
\]
Subtracting these congruences, we would obtain
\[
s_i\equiv s_j\pmod p,
\]
so that $p\mid s_i-s_j$. Consequently, $p\mid\Delta$,
contrary to the definition of $P_*$. Thus the roots
are pairwise distinct, as required.
\end{proof}

Introduce
\[
 B_0=\sum_{\substack{d\mid P_0\\d>1}}
       \frac1{\varphi(d)\log d},\qquad
 B=\sum_{\substack{d\mid P_*\\d>1}}
       \frac1{\varphi(d)\log d},\qquad
 U_0=\sum_{d\mid P_0}\frac1{\varphi(d)}.
\]

\begin{lemma}[Divisor-sum estimates]\label{lem:divisor-sums}
For the sums defined above,
\begin{equation}\label{eq:full-estimates}
 B_0\asymp1,\qquad
 \sum_{\substack{d\mid P_0\\d>1}}
       \frac1{\varphi(d)(\log d)^2}\ll\frac1{(\log x)^2},
 \qquad U_0\asymp\frac{\log y}{\log z}.
\end{equation}
Moreover, $B\asymp1$. All implied constants are absolute.
\end{lemma}

\begin{proof}
We shall use the following standard forms of Mertens' estimates:
\[
 \prod_{q\le t}\left(1-\frac1q\right)^{-1}
 \asymp\log t,\qquad
 \sum_{q\le t}\frac{\log q}{q}\ll\log t,
\]
and
\[
 \sum_{q\le t}\frac1q
 =\log\log t+M+O\!\left(\frac1{\log t}\right),
\]
where $M$ is an absolute constant and the letters $p$ and $q$
denote primes in these Mertens estimates.

\medskip\noindent
\textit{The estimate for $U_0$.}
Since $P_0$ is squarefree, multiplicativity gives
\[
 U_0
 =\sum_{d\mid P_0}\frac1{\varphi(d)}
 =\prod_{z<p\le y}\left(1+\frac1{p-1}\right)
 =\prod_{z<p\le y}\left(1-\frac1p\right)^{-1}.
\]
Consequently, Mertens' product estimate yields
\[
 U_0
 =\frac{\displaystyle
          \prod_{p\le y}(1-1/p)^{-1}}
        {\displaystyle
          \prod_{p\le z}(1-1/p)^{-1}}
 \asymp\frac{\log y}{\log z}.
\]

\medskip\noindent
\textit{Grouping divisors by their largest prime factor.}
For each prime $p\in(z,y]$, we put
\[
 P_{<p}=\prod_{z<q<p}q,\qquad
 V_p=\sum_{e\mid P_{<p}}\frac1{\varphi(e)}.
\]
Again by multiplicativity and Mertens' product estimate,
\[
 V_p=\prod_{z<q<p}\left(1-\frac1q\right)^{-1}
 \asymp\frac{\log p}{\log z},
\]
uniformly for $z<p\le y$. Indeed,
\[
 V_p
 =\left(1-\frac1p\right)
   \frac{\displaystyle
          \prod_{q\le p}(1-1/q)^{-1}}
        {\displaystyle
          \prod_{q\le z}(1-1/q)^{-1}}.
\]

Every divisor $d>1$ of $P_0$ has a unique representation
$d=pe$, where $p$ is its largest prime factor and
$e\mid P_{<p}$. Since $(p,e)=1$, we have
$\varphi(pe)=(p-1)\varphi(e)$. Thus
\[
 B_0
 =\sum_{z<p\le y}\frac1{p-1}
   \sum_{e\mid P_{<p}}
       \frac1{\varphi(e)\log(pe)}.
\]

\medskip\noindent
\textit{Upper and lower bounds for the inner sum.}
Let
\[
 T_p=\sum_{e\mid P_{<p}}
          \frac1{\varphi(e)\log(pe)}.
\]
Hence,
\begin{equation}\label{B_VIEW}
 B_0 = \sum_{z<p \le y} \frac{T_p}{p-1}.
\end{equation}

Let us show that
\[
T_p \asymp \frac{V_p}{\log p}.
\]
The inequality $\log(pe)\ge\log p$ immediately gives
\[
 T_p\le\frac{V_p}{\log p}.
\]

For the lower bound, we first evaluate a logarithmic moment.
Fix a prime $q\mid P_{<p}$. Since $P_{<p}$ is squarefree,
every divisor $e\mid P_{<p}$ divisible by $q$ can be written
uniquely as
\[
 e=qf,\qquad f\mid P_{<p}/q,\qquad (q,f)=1.
\]
By the multiplicativity of $\varphi$, we have
$\varphi(qf)=(q-1)\varphi(f)$. Therefore
\[
 \sum_{\substack{e\mid P_{<p}\\q\mid e}}
       \frac1{\varphi(e)}
 =\frac1{q-1}
   \sum_{f\mid P_{<p}/q}\frac1{\varphi(f)}.
\]
We see that
\[
 \sum_{f\mid P_{<p}/q}\frac1{\varphi(f)}
 =\prod_{\substack{z<r<p\\r\in\PP\\r\ne q}}
       \left(1+\frac1{r-1}\right).
\]
On the other hand, separating the Euler factor corresponding
to $q$ in the product for $V_p$, we obtain
\[
 \begin{aligned}
 V_p
 &=\prod_{\substack{z<r<p\\r\in\PP}}\left(1+\frac1{r-1}\right)\\
 &=\left(1+\frac1{q-1}\right)
   \prod_{\substack{z<r<p\\r\in\PP\\r\ne q}}
       \left(1+\frac1{r-1}\right)\\
 &=\frac{q}{q-1}
   \prod_{\substack{z<r<p\\r\in\PP\\r\ne q}}
       \left(1+\frac1{r-1}\right).
 \end{aligned}
\]
Consequently,
\[
 \begin{aligned}
 \sum_{\substack{e\mid P_{<p}\\q\mid e}}
       \frac1{\varphi(e)}
 &=\frac1{q-1}
   \prod_{\substack{z<r<p\\r\in\PP\\r\ne q}}
       \left(1+\frac1{r-1}\right)\\
 &=\frac1{q-1}\cdot\frac{q-1}{q}\,V_p
 =\frac{V_p}{q}.
 \end{aligned}
\]
Since all divisors of $P_{<p}$ are squarefree, it follows that
\[
 \begin{aligned}
 \sum_{e\mid P_{<p}}\frac{\log e}{\varphi(e)}
 &=\sum_{e\mid P_{<p}}\frac{1}{\varphi(e)}\sum_{q|e} \log q\\
 &=\sum_{z<q<p}\log q
   \sum_{\substack{e\mid P_{<p}\\q\mid e}}
       \frac1{\varphi(e)}
 =V_p\sum_{z<q<p}\frac{\log q}{q}.
  \end{aligned}
\]
Therefore
\[
 \sum_{e\mid P_{<p}}\frac{\log(pe)}{\varphi(e)}
 =V_p\left(\log p+
          \sum_{z<q<p}\frac{\log q}{q}\right)
 \ll V_p\log p.
\]

Applying the Cauchy--Schwarz inequality, we obtain
\[
 \begin{aligned}
 V_p^2
 &=
 \left(
   \sum_{e\mid P_{<p}}
   \sqrt{\frac1{\varphi(e)\log(pe)}}
   \sqrt{\frac{\log(pe)}{\varphi(e)}}
 \right)^2\\
 &\le
 T_p\sum_{e\mid P_{<p}}
          \frac{\log(pe)}{\varphi(e)}
 \ll T_pV_p\log p.
 \end{aligned}
\]
Hence
\[
 T_p\gg\frac{V_p}{\log p}.
\]
Combining the two bounds, we conclude that
\[
 T_p\asymp\frac{V_p}{\log p}
 \asymp\frac1{\log z},
\]
uniformly for $z<p\le y$.

\medskip\noindent
\textit{The estimate for $B_0$.}
By \eqref{B_VIEW} we have
\[
 B_0
 \asymp\frac1{\log z}
       \sum_{z<p\le y}\frac1{p-1}.
\]
Since
\[
 \sum_{z<p\le y}
 \left(\frac1{p-1}-\frac1p\right)
 \ll\sum_{n>z}\frac1{n^2}
 \ll\frac1z,
\]
Mertens' estimate implies
\[
 \sum_{z<p\le y}\frac1{p-1}
 =\log\!\left(\frac{\log y}{\log z}\right)
  +O\!\left(\frac1{\log z}\right).
\]
For our choice of parameters,
\[
 \frac{\log y}{\log z}
 =\frac{x}{6(\log x)^6},
\]
and consequently
\[
 \frac1{\log z}
 \log\!\left(\frac{\log y}{\log z}\right)
 =
 \frac{\log x-6\log\log x-\log6}{6\log x}
 \longrightarrow\frac16.
\]
It follows that $B_0\asymp1$.

\medskip\noindent
\textit{The sum with $(\log d)^2$ in the denominator.}
Grouping divisors by their largest prime factor once more,
we obtain
\[
 \begin{aligned}
 \sum_{\substack{d\mid P_0\\d>1}}
       \frac1{\varphi(d)(\log d)^2}
 &=
 \sum_{z<p\le y}\frac1{p-1}
   \sum_{e\mid P_{<p}}
       \frac1{\varphi(e)(\log(pe))^2}\\
 &\le
 \sum_{z<p\le y}
       \frac{V_p}{(p-1)(\log p)^2}\\
 &\ll
 \frac1{\log z}
 \sum_{z<p\le y}\frac1{p\log p}.
 \end{aligned}
\]

To estimate the last prime sum, we define
\[
 H_z(t)=\sum_{z<p\le t}\frac{\log p}{p}
 \qquad(t\ge z).
\]
Mertens' estimate gives $H_z(t)\ll\log t$.
Partial summation therefore yields
\[
 \begin{aligned}
 \sum_{p>z}\frac1{p\log p}
 &=
 \lim_{R\to\infty}
 \left(
   \frac{H_z(R)}{(\log R)^2}
   +2\int_z^R
       \frac{H_z(t)}{t(\log t)^3}\,dt
 \right)\\
 &\ll
 \int_z^\infty\frac{dt}{t(\log t)^2}
 =\frac1{\log z}.
 \end{aligned}
\]
Thus
\[
 \sum_{\substack{d\mid P_0\\d>1}}
       \frac1{\varphi(d)(\log d)^2}
 \ll\frac1{(\log z)^2}
 \ll\frac1{(\log x)^2},
\]
because $\log z=6\log x$. This proves all three estimates
in~\eqref{eq:full-estimates}.

For every prime $p\mid P_0$, we have
\begin{equation}\label{eq:U0/p}
 \sum_{\substack{d\mid P_0\\p\mid d}}\frac1{\varphi(d)}
 =\frac{U_0}{p}.
\end{equation}
Indeed, since $P_0$ is squarefree, every divisor $d\mid P_0$
divisible by $p$ can be written uniquely as $d=pe$, where
$e\mid P_0/p$ and $(p,e)=1$. Hence
$\varphi(pe)=(p-1)\varphi(e)$, and therefore
\[
 \sum_{\substack{d\mid P_0\\p\mid d}}\frac1{\varphi(d)}
 =\frac1{p-1}\sum_{e\mid P_0/p}\frac1{\varphi(e)}.
\]
On the other hand, separating the divisors of $P_0$ according
to whether they are divisible by $p$, we obtain
\[
 \begin{aligned}
 U_0
 &=\sum_{e\mid P_0/p}\frac1{\varphi(e)}
   +\sum_{e\mid P_0/p}\frac1{\varphi(pe)}\\
 &=\left(1+\frac1{p-1}\right)
   \sum_{e\mid P_0/p}\frac1{\varphi(e)}
 =\frac{p}{p-1}\sum_{e\mid P_0/p}\frac1{\varphi(e)}.
 \end{aligned}
\]
Consequently,
\[
 \sum_{\substack{d\mid P_0\\p\mid d}}\frac1{\varphi(d)}
 =\frac1{p-1}\cdot\frac{p-1}{p}\,U_0
 =\frac{U_0}{p},
\]
and \eqref{eq:U0/p} is proved.

Let
\[
 \mathcal E=\{p\in\PP:z<p\le y,\ p\mid\Delta\}
\]
be the set of removed primes. Since $P_*\mid P_0$, we have
\[
 B_0-B
 =\sum_{\substack{d\mid P_0\\d\nmid P_*}}
       \frac1{\varphi(d)\log d}.
\]
Every divisor occurring in this sum is divisible by at least
one prime in $\mathcal E$. All summands are nonnegative, so
\[
 0\le B_0-B
 \le\sum_{p\in\mathcal E}
      \sum_{\substack{d\mid P_0\\p\mid d}}
          \frac1{\varphi(d)\log d}.
\]
A divisor containing several primes from $\mathcal E$ is counted
more than once on the right, which only increases the sum.

For each term in the inner sum, we have $d\ge p>z$, and hence
$1/\log d\le1/\log z$. Using \eqref{eq:U0/p}, \eqref{eq:exceptional}, \eqref{eq:full-estimates}, we obtain
\begin{equation}\label{eq:B-loss}
 \begin{aligned}
 0\le B_0-B
 &\le\frac1{\log z}
      \sum_{p\in\mathcal E}
      \sum_{\substack{d\mid P_0\\p\mid d}}
          \frac1{\varphi(d)}\\
 &=\frac{U_0}{\log z}
      \sum_{p\in\mathcal E}\frac1p\\
 &\ll\frac{\log y}{(\log z)^2}
      \frac1{x^3\log x}\\
 &\ll\frac1{x^2(\log x)^8}=o(1).
 \end{aligned}
\end{equation}
Since $B_0\asymp1$ and $B_0-B=o(1)$ by~\eqref{eq:B-loss}, we conclude that
$B\asymp1$. In particular, $P_*>1$ for sufficiently large $x$.
\end{proof}

\subsection{Estimates for the coefficients}

Put
\[
 \begin{gathered}
 a(1)=1,\qquad a(d)=-\frac1{B\log d}\quad(d\mid P_*,\ d>1),\\
 A_\gamma=\sum_{d\mid P_*}\frac{a(d)^2d^\gamma}{\varphi(d)},
 \qquad A=A_0.
 \end{gathered}
\]

\begin{lemma}[Coefficient estimates]\label{lem:coefficient-bounds}
For the coefficients $a(d)$ defined above, we have
\begin{equation}\label{eq:basic-coefficients}
 \sum_{d\mid P_*}\frac{a(d)}{\varphi(d)}=0,
 \qquad |a(d)|\le1,
 \qquad A=1+O((\log x)^{-2}).
\end{equation}
There is an absolute constant $C>0$ such that
\begin{equation}\label{eq:tilt}
 \sum_{\substack{d\mid P_*\\d>1}}
       \frac{|a(d)|d^\gamma}{\varphi(d)}\le C,
 \qquad A_\gamma\le A+C\gamma
 \qquad(0\le\gamma\le2\beta).
\end{equation}
Finally, for every prime $p\mid P_*$,
\begin{equation}\label{eq:prime-coefficients}
 \sum_{\substack{d\mid P_*\\p\mid d}}
       \frac{|a(d)|}{\varphi(d)}
  \ll\frac1p.
\end{equation}
All the implied constants are absolute.
\end{lemma}

\begin{proof}
We first check the cancellation identity and the size of the coefficients.
Here and below, a statement about $a(d)$ is understood to concern
divisors $d$ of $P_*$. Since $\varphi(1)=1$, the definition of $B$
and that of $a(d)$ give the exact cancellation
\[
 \sum_{d\mid P_*}\frac{a(d)}{\varphi(d)}
 =1-\frac1B\sum_{\substack{d\mid P_*\\d>1}}
                  \frac1{\varphi(d)\log d}
 =1-\frac BB=0.
\]
Next, $|a(1)|=1$. If $d>1$ divides $P_*$, then $d$ has a prime
factor greater than $z$, and consequently $d>z$. We have already
proved that $B\asymp1$, so there is an absolute constant $b_1>0$
such that $B\ge b_1$ for all sufficiently large $x$. Therefore
\[
 |a(d)|=\frac1{B\log d}
 \le\frac1{b_1\log z}
 =\frac1{6b_1\log x}\le1
 \qquad(d>1)
\]
once $x$ is sufficiently large.

Finally, separating the term $d=1$ in $A$ and using $P_*\mid P_0$,
we obtain
\[
 \begin{aligned}
 0\le A-1
 &=\frac1{B^2}\sum_{\substack{d\mid P_*\\d>1}}
                  \frac1{\varphi(d)(\log d)^2}\\
 &\le\frac1{B^2}\sum_{\substack{d\mid P_0\\d>1}}
                  \frac1{\varphi(d)(\log d)^2}
 \ll\frac1{(\log x)^2}.
 \end{aligned}
\]
The last estimate uses~\eqref{eq:full-estimates} and $B\asymp1$.
This proves~\eqref{eq:basic-coefficients}; in particular, $A\ge1$.

\medskip\noindent\textit{Estimates with a small power of the divisor.}
For $0\le\gamma\le2\beta$, we define
\[
 U_\gamma=\sum_{d\mid P_0}\frac{d^\gamma}{\varphi(d)}.
\]
The function $d^\gamma/\varphi(d)$ is multiplicative, and $P_0$ is
squarefree. Expanding its Euler product and then dividing by the
corresponding product for $U_0$, we find
\[
 \begin{aligned}
 U_\gamma&=\prod_{z<p\le y}
                    \left(1+\frac{p^\gamma}{p-1}\right),\\
 \frac{U_\gamma}{U_0}
 &=\prod_{z<p\le y}
     \frac{1+p^\gamma/(p-1)}{1+1/(p-1)}
 =\prod_{z<p\le y}\left(1+\frac{p^\gamma-1}{p}\right).
 \end{aligned}
\]
For every prime in this product, $0\le\gamma\log p\le2$.
The identity $e^v-1=\int_0^v e^s\,ds$ therefore gives
\[
 0\le p^\gamma-1
 =e^{\gamma\log p}-1
 \le e^2\gamma\log p.
\]
Applying $1+h\le e^h$ for $h\ge0$ to the Euler factors, and then
using Mertens' estimate, we obtain
\[
 \frac{U_\gamma}{U_0}
 \le\exp\!\left(\sum_{z<p\le y}\frac{p^\gamma-1}{p}\right)
 \le\exp\!\left(e^2\gamma\sum_{p\le y}\frac{\log p}{p}\right)
 \le\exp(C\gamma\log y)\ll1.
\]
Consequently, $U_\gamma\ll U_0\ll\log y/\log z$.

We next estimate the sum involving $|a(d)|$ by separating the
ranges $1<d\le y$ and $d>y$. In the first range, $d^\gamma\le e^2$,
and the definition of $B$ gives
\[
 \begin{aligned}
 \sum_{\substack{d\mid P_*\\1<d\le y}}
       \frac{|a(d)|d^\gamma}{\varphi(d)}
 &\le\frac{e^2}{B}
       \sum_{\substack{d\mid P_*\\1<d\le y}}
            \frac1{\varphi(d)\log d}\\
 &\le\frac{e^2}{B}
       \sum_{\substack{d\mid P_*\\d>1}}
            \frac1{\varphi(d)\log d}=e^2.
 \end{aligned}
\]
In the second range, $1/\log d\le1/\log y$. Since $P_*\mid P_0$,
we have
\[
 \begin{aligned}
 \sum_{\substack{d\mid P_*\\d>y}}
       \frac{|a(d)|d^\gamma}{\varphi(d)}
 &\le\frac1{B\log y}
       \sum_{d\mid P_*}\frac{d^\gamma}{\varphi(d)}\\
 &\le\frac{U_\gamma}{B\log y}
 \ll\frac1{\log z}\ll 1.
 \end{aligned}
\]
Adding the two estimates proves
\[
 \sum_{\substack{d\mid P_*\\d>1}}
       \frac{|a(d)|d^\gamma}{\varphi(d)}
 \ll 1.
\]

To control $A_\gamma$, observe that for $d\ge1$,
\[
 d^\gamma-1
 =\int_0^\gamma d^s\log d\,ds
 \le\gamma d^\gamma\log d.
\]
Also, for $d>1$, the definition of $a(d)$ implies
$a(d)^2\log d=|a(d)|/B$. It follows that
\[
 \begin{aligned}
 0\le A_\gamma-A
 &=\sum_{\substack{d\mid P_*\\d>1}}
       \frac{a(d)^2(d^\gamma-1)}{\varphi(d)}\\
 &\le\gamma\sum_{\substack{d\mid P_*\\d>1}}
       \frac{a(d)^2d^\gamma\log d}{\varphi(d)}\\
 &=\frac\gamma B\sum_{\substack{d\mid P_*\\d>1}}
       \frac{|a(d)|d^\gamma}{\varphi(d)}
 \ll\gamma.
 \end{aligned}
\]
This proves \eqref{eq:tilt} with an absolute constant $C$.
For later use, fix an absolute constant $C_1\ge1$ such that
$A_\gamma\le A+C_1\gamma$ throughout this range.

\medskip\noindent\textit{Divisors containing a prescribed prime.}
Fix a prime $p\mid P_*$. Every divisor $d\mid P_*$ with $p\mid d$
has a unique representation $d=pe$, where $e\mid P_*/p$ and
$(p,e)=1$. Separating $e=1$ and using
$\varphi(pe)=(p-1)\varphi(e)$, we obtain
\[
 \begin{aligned}
 \sum_{\substack{d\mid P_*\\p\mid d}}
       \frac{|a(d)|}{\varphi(d)}
 &=\frac1{B(p-1)}
   \sum_{e\mid P_*/p}\frac1{\varphi(e)\log(pe)}\\
 &=\frac1{B(p-1)}\left(
      \frac1{\log p}
      +\sum_{\substack{e\mid P_*/p\\e>1}}
          \frac1{\varphi(e)\log(pe)}\right).
 \end{aligned}
\]
For $e>1$, we have $\log(pe)\ge\log e$, while every divisor of
$P_*/p$ is also a divisor of $P_*$. Hence
\[
 \sum_{\substack{e\mid P_*/p\\e>1}}
       \frac1{\varphi(e)\log(pe)}
 \le\sum_{\substack{e\mid P_*\\e>1}}
       \frac1{\varphi(e)\log e}=B.
\]
Substitution proves \eqref{eq:prime-coefficients}.
Indeed, $B$ is bounded below by an absolute positive constant,
$\log p\ge\log z$, and $1/(p-1)\le2/p$.
\end{proof}

\subsection{Local factors and their moments}

For $p\mid P_*$ and $d\mid P_*$, we define
\[
 \psi_{i,p}(t)=\frac{1-p\ind_{p\mid L_i(t)}}{p-1},
 \qquad
 \psi_{i,d}(t)=\prod_{p\mid d}\psi_{i,p}(t),
 \qquad \psi_{i,1}(t)=1.
\]
For a $P_*$-periodic function $F$, write
\[
 \E F=\frac1{P_*}\sum_{t=0}^{P_*-1}F(t).
\]

\begin{lemma}[Local moments]\label{lem:local-factor-moments}
Under the complete average $\E$, the residue coordinates modulo
all distinct primes dividing $P_*$ are independent. The local
factors satisfy
\begin{equation}\label{eq:local-moments}
 \E\psi_{i,p}=0,\qquad
 \E\psi_{i,p}^2=\frac1{p-1},\qquad
 \E(\psi_{i,p}\psi_{j,p})=-\frac1{(p-1)^2}\quad(i\ne j).
\end{equation}
\end{lemma}

\begin{proof}
Fix a prime $p\mid P_*$. For any function $f$ that is periodic modulo
$p$, grouping the terms according to their residues modulo
$p$ gives
\[
 \begin{aligned}
 \E f
 &=\frac1{P_*}\sum_{t=0}^{P_*-1}f(t)\\
 &=\frac1{P_*}\sum_{r=0}^{p-1}
       \sum_{h=0}^{P_*/p-1}f(r+hp)\\
 &=\frac1{P_*}\sum_{r=0}^{p-1}
       \frac{P_*}{p}f(r)
 =\frac1p\sum_{r=0}^{p-1}f(r).
 \end{aligned}
\]
In particular, this identity applies to the local factors
$\psi_{i,p}$ and their products.

Let $f_p(t)$ be any function depending only
on $t\bmod p$. Since $P_*$ is squarefree, the Chinese remainder
theorem gives a bijection
\[
 \mathbb Z/P_*\mathbb Z
 \longrightarrow
 \prod_{p\mid P_*}\mathbb Z/p\mathbb Z,
 \qquad
 t\longmapsto(t\bmod p)_{p\mid P_*}.
\]
Hence the average of a product of such functions factors as
\[
 \begin{aligned}
 \E\prod_{p\mid P_*}f_p(t)
 &=\frac1{P_*}
   \sum_{\substack{0\le a_p<p\\\text{for each }p\mid P_*}}
      \prod_{p\mid P_*}f_p(a_p)\\
 &=\prod_{p\mid P_*}
      \left(\frac1p\sum_{a=0}^{p-1}f_p(a)\right)
 =\prod_{p\mid P_*}\E f_p(t).
 \end{aligned}
\]
Each function $\psi_{i,p}(t)$ depends only on $t\bmod p$.
Therefore, to average a product of these functions, we first
group together all factors associated with the same prime $p$,
and then factor the average over the distinct primes.
We now compute the local moments needed for these factorizations.

Since $(Q,p)=1$, multiplication by $Q$ permutes the residue
classes modulo $p$. Therefore the congruence
\[
 L_i(t)=b+Qt+s_i\equiv0\pmod p
\]
has exactly one solution modulo $p$. By the definition of
$\psi_{i,p}$,
\[
 \psi_{i,p}(t)=
 \begin{cases}
  -1,&p\mid L_i(t),\\[1mm]
  \dfrac1{p-1},&p\nmid L_i(t).
 \end{cases}
\]
Hence $\psi_{i,p}$ takes the value $-1$ on one residue class
modulo $p$, of weight $1/p$, and the value $1/(p-1)$ on the
other $p-1$ classes, of total weight $(p-1)/p$.
Direct calculation gives
\[
 \E\psi_{i,p}
 =\frac1p\left(-1+(p-1)\frac1{p-1}\right)=0
\]
and
\[
 \E\psi_{i,p}^2
 =\frac1p\left(1+(p-1)\frac1{(p-1)^2}\right)
 =\frac1{p-1}.
\]
If $i\ne j$, the roots of $L_i$ and $L_j$ modulo $p$ are distinct:
a common root would imply $p\mid s_i-s_j$, contrary to
$p\nmid\Delta$. On each of these two root classes,
$\psi_{i,p}\psi_{j,p}=-1/(p-1)$; on the remaining $p-2$ classes,
the product is $1/(p-1)^2$. Therefore
\[
 \E(\psi_{i,p}\psi_{j,p})
 =\frac1p\left(-\frac2{p-1}+\frac{p-2}{(p-1)^2}\right)
 =-\frac1{(p-1)^2}.
\]
We have proved the local moment identities in \eqref{eq:local-moments}.
The independence established above applies only to distinct prime
moduli. For a fixed prime $p$, the functions $\psi_{i,p}$ and
$\psi_{j,p}$ need not be independent. Indeed, for $i\ne j$,
\eqref{eq:local-moments} gives
\[
 \E(\psi_{i,p}\psi_{j,p})
 =-\frac1{(p-1)^2}
 \ne 0
 =(\E\psi_{i,p})(\E\psi_{j,p}).
\]
\end{proof}

\subsection{The truncated weight and finite averages}

For $0\le j\le k$, let
\[
 \mathcal D_j=\{(r_1,\ldots,r_j):r_i\mid P_*,\ (r_i,r_\ell)=1
 \ (i\ne\ell),\ r_1\cdots r_j\le D\},
\]
where $\mathcal D_0$ contains only the empty tuple, whose product
is $1$. The coprimality condition ensures that each prime occurs
in at most one coordinate, while the product condition bounds the
moduli that arise when the weight is expanded. We define
\[
 R(t)=\sum_{\bm r\in\mathcal D_k}
             \prod_{i=1}^k a(r_i)\psi_{i,r_i}(t),
 \qquad w(t)=R(t)^2,
\]
and
\begin{equation}\label{S_DEF}
 S_1=\frac1T\sum_{t=1}^T w(t),\qquad
 S_2=\frac1T\sum_{t=1}^T w(t)
                \sum_{i=1}^k\ind_{(L_i(t),P_*)=1}.
\end{equation}
We call $L_i(t)$ \emph{uncovered} if $(L_i(t),P_*)=1$. The inner sum in
$S_2$ is the number of uncovered forms, a nonnegative integer.
If it were at least $1$ for every $t$, the nonnegativity of $w(t)$
would imply $S_2\ge S_1$. Thus $S_2<S_1$ will prove that some
$1\le t\le T$ has no uncovered form.

For a squarefree divisor $r\mid P_*$, expanding the product defining
$\psi_{i,r}$ gives
\[
 \begin{aligned}
 \psi_{i,r}(t)
 &=\frac1{\varphi(r)}
           \prod_{p\mid r}(1-p\ind_{p\mid L_i(t)})\\
 &=\frac1{\varphi(r)}
           \sum_{d\mid r}(-1)^{\omega(d)}d\,
                       \ind_{d\mid L_i(t)}.
 \end{aligned}
\]
Here we used the fact that
\[
 \prod_{p|d}\ind_{p\mid L_i(t)} =
 \ind_{d\mid L_i(t)}.
\]
If $d_i\mid r_i$ for a tuple
$\bm r\in\mathcal D_k$, then the $d_i$ are pairwise coprime and
$d_1\cdots d_k\le r_1\cdots r_k\le D$. Hence, after collecting
coefficients, the expansion has the form
\begin{equation}\label{R_VIEW}
 R(t)=\sum_{\bm d\in\mathcal D_k}\lambda_{\bm d}
                         \prod_{i=1}^k\ind_{d_i\mid L_i(t)},
\end{equation}
where
\[
 \lambda_{\bm d}= \sum_{\substack{\bm r\in\mathcal D_k\\ d_i | r_i}}
 \prod_{i=1}^k \left(\frac{a(r_i)}{\varphi (r_i)} (-1)^{\omega (d_i)} d_i\right).
\]

\begin{lemma}[A finite-average comparison]\label{lem:finite-averages}
The coefficients in \eqref{R_VIEW} satisfy
\[
 \sum_{\bm d\in\mathcal D_k}|\lambda_{\bm d}|\le D^{3/2}.
\]
Furthermore, for every subset $\mathcal J\subseteq\mathcal D_k$,
the quadratic averages of the corresponding partial sum of
\eqref{R_VIEW}, taken over $1\le t\le T$ and modulo $P_*$,
differ by at most $D^3/T$. Both estimates are uniform in $b$
and in the allowed shifts.
\end{lemma}

\begin{proof}
We estimate the sum of absolute coefficients before collecting
equal indicator products. The triangle inequality gives
\[
 \begin{aligned}
 \sum_{\bm d\in\mathcal D_k}|\lambda_{\bm d}|
 &\le\sum_{\bm r\in\mathcal D_k}
       \prod_{i=1}^k\left(
          \frac{|a(r_i)|}{\varphi(r_i)}\sum_{d_i\mid r_i}d_i
       \right)\\
 &=\sum_{\bm r\in\mathcal D_k}
       \prod_{i=1}^k\frac{|a(r_i)|\sigma_1(r_i)}{\varphi(r_i)}\\
 &\le\sum_{\bm r\in\mathcal D_k}
       \frac{\sigma_1(r_1\cdots r_k)}{\varphi(r_1\cdots r_k)}.
 \end{aligned}
\]
The last step uses $|a(r_i)|\le1$ and the multiplicativity of
$\sigma_1$ and $\varphi$ on pairwise coprime arguments.

Fix a divisor $n\mid P_*$ with $n\le D$. We count the ordered
tuples $(r_1,\ldots,r_k)\in\mathcal D_k$ satisfying
$r_1\cdots r_k=n$. Suppose first that $n>1$. Since $P_*$
is squarefree, we may write
\[
 n=p_1\cdots p_\ell,\qquad \ell=\omega(n),
\]
where the primes $p_1,\ldots,p_\ell$ are distinct.

In every such tuple, each prime $p_j$ divides exactly one
coordinate. Indeed, it divides at least one coordinate because
$r_1\cdots r_k=n$, and it cannot divide two coordinates because
the coordinates are pairwise coprime. Thus the tuple determines
a map
\[
 f:\{1,\ldots,\ell\}\longrightarrow\{1,\ldots,k\},
\]
where $f(j)=i$ means that $p_j\mid r_i$.

Conversely, every such map determines a tuple by
\[
 r_i=\prod_{\substack{1\le j\le\ell\\f(j)=i}}p_j
 \qquad(1\le i\le k),
\]
with an empty product interpreted as $1$. These coordinates
are pairwise coprime divisors of $P_*$, and their product is
$n\le D$. Hence the resulting tuple belongs to $\mathcal D_k$.
The two constructions are inverse to each other, so this is
a bijection.

For each of the $\ell$ primes, there are $k$ independent
choices of its coordinate. Therefore
\[
 \#\{(r_1,\ldots,r_k)\in\mathcal D_k:r_1\cdots r_k=n\}
 =k^\ell=k^{\omega(n)}.
\]
For $n=1$, the only possible tuple is $(1,\ldots,1)$.
Since $\omega(1)=0$ and $k\ge1$, the same formula gives
$k^{\omega(1)}=1$.

It follows that
\[
 \sum_{\bm d\in\mathcal D_k}|\lambda_{\bm d}|
 \le\sum_{\substack{n\mid P_*\\n\le D}}
             k^{\omega(n)}\frac{\sigma_1(n)}{\varphi(n)}.
\]

We now bound both factors in each summand. All prime divisors of
$n$ exceed $z$, so
\begin{equation}\label{w(n).calc}
 \omega(n)\log z\le\sum_{p\mid n}\log p
 =\log n\le\log D.
\end{equation}
Since $n$ is squarefree and $p>z\ge3$, we also have
\[
 \begin{aligned}
 \frac{\sigma_1(n)}{\varphi(n)}
 &=\prod_{p\mid n}\frac{p+1}{p-1}
 =\prod_{p\mid n}\left(1+\frac2{p-1}\right)\\
 &\le\exp\!\left(\sum_{p\mid n}\frac2{p-1}\right)
 \le\exp\!\left(\frac{3\omega(n)}z\right)
 \le\exp\!\left(\frac{3\log D}{z\log z}\right)\le2
 \end{aligned}
\]
for sufficiently large $x$. The estimate includes $n=1$.
Furthermore, $k\ge1$, so
\[
 k^{\omega(n)}
 \le\exp\!\left(\frac{\log D\log k}{\log z}\right)
 =D^{\log k/\log z}.
\]
Since $k\le\eta x\le x$ and $\log z=6\log x$,
we have $\log k/\log z\le1/6$. Thus
\begin{equation}\label{eq:l1}
 \sum_{\bm d\in\mathcal D_k}|\lambda_{\bm d}|
 \le\sum_{\substack{n\mid P_*\\n\le D}}
             k^{\omega(n)}\frac{\sigma_1(n)}{\varphi(n)}
 \le2D^{1+\log k/\log z}\le D^{3/2}.
\end{equation}
The last inequality follows from $2D^{7/6}\le D^{3/2}$ for
$D\ge8$, which holds once $x$ is sufficiently large.

\medskip\noindent\textit{Comparison of the two averages.}
For $\bm d\in\mathcal D_k$, write
$J_{\bm d}(t)=\prod_{i=1}^k\ind_{d_i\mid L_i(t)}$.
Expanding the square of $R$ gives
\[
 R(t)^2=\sum_{\bm d,\bm e\in\mathcal D_k}
               \lambda_{\bm d}\lambda_{\bm e}
               J_{\bm d}(t)J_{\bm e}(t).
\]

Fix two tuples $\bm d,\bm e\in\mathcal D_k$. By definition,
\[
 J_{\bm d}(t)J_{\bm e}(t)
 =\prod_{i=1}^k
       \ind_{d_i\mid L_i(t)}\ind_{e_i\mid L_i(t)}.
\]
Thus this product equals $1$ if and only if all the conditions
\[
 d_i\mid L_i(t),\qquad e_i\mid L_i(t)
 \qquad(1\le i\le k)
\]
hold simultaneously. Put
\[
 q=\operatorname{lcm}(d_1,\ldots,d_k,e_1,\ldots,e_k),
\]
where $\operatorname{lcm}$ denotes the least common multiple.
Since all the coordinates divide $P_*$, we have $q\mid P_*$.

Suppose first that $(d_i,e_j)>1$ for some distinct indices
$i$ and $j$. Choose a prime $p\mid(d_i,e_j)$. If the product
$J_{\bm d}(t)J_{\bm e}(t)$ were equal to $1$, then
\[
 p\mid L_i(t),\qquad p\mid L_j(t),
\]
and consequently
\[
 p\mid L_i(t)-L_j(t)=s_i-s_j.
\]
This is impossible: $p\mid P_*$, whereas every prime divisor
of a difference $s_i-s_j$ has been excluded from $P_*$.
Therefore
\[
 J_{\bm d}(t)J_{\bm e}(t)=0
 \qquad\text{for every }t\in\cZ.
\]

Suppose now that
\[
 (d_i,e_j)=1\qquad\text{for all }i\ne j.
\]
For each $i$, define
\[
 m_i=\operatorname{lcm}(d_i,e_i).
\]
The two divisibility conditions for the $i$th form are
equivalent to a single condition:
\[
 d_i\mid L_i(t)\ \text{and}\ e_i\mid L_i(t)
 \quad\Longleftrightarrow\quad
 m_i\mid L_i(t).
\]
Moreover, since $\bm d,\bm e\in\mathcal D_k$, their coordinates
are pairwise coprime within each tuple. Together with the
assumed cross-coprimality, this gives
\[
 (d_i,d_j)=(e_i,e_j)=(d_i,e_j)=(e_i,d_j)=1
 \qquad(i\ne j).
\]
Hence $(m_i,m_j)=1$ whenever $i\ne j$, and therefore
\[
 m_1\cdots m_k
 =\operatorname{lcm}(m_1,\ldots,m_k)
 =q.
\]

Each $m_i$ divides $P_*$, so $(Q,m_i)=1$. It follows that
the congruence
\[
 b+Qt+s_i\equiv0\pmod{m_i}
\]
determines exactly one residue class modulo $m_i$.
When $m_i=1$, this condition is satisfied by every integer.
Since the moduli $m_i$ are pairwise coprime, the Chinese
remainder theorem shows that the entire system is equivalent
to one congruence
\[
 t\equiv a\pmod q
\]
for a uniquely determined residue class $a\bmod q$.
Consequently,
\[
 J_{\bm d}(t)J_{\bm e}(t)=
 \begin{cases}
  0,
  &\text{if }(d_i,e_j)>1\text{ for some }i\ne j,\\[1mm]
  \ind_{t\equiv a\pmod q},
  &\text{if }(d_i,e_j)=1\text{ for all }i\ne j.
 \end{cases}
\]
If $q=1$, the indicator in the second case is identically $1$.

We now compare the complete average with the average over
$1\le t\le T$. In the first case, both averages vanish.
In the second case, since $q\mid P_*$, exactly $P_*/q$
integers in $\{0,\ldots,P_*-1\}$ belong to the class
$a\bmod q$. Thus
\[
 \E(J_{\bm d}J_{\bm e})
 =\frac1{P_*}
   \#\{0\le t<P_*:t\equiv a\pmod q\}
 =\frac1q.
\]

To estimate the number of representatives in the shorter
interval, write
\[
 T=\ell q+h,\qquad 0\le h<q.
\]
Each of the $\ell$ complete blocks of $q$ consecutive
integers contains exactly one representative of $a\bmod q$,
and the remaining block contains at most one. Hence
\[
 \#\{1\le t\le T:t\equiv a\pmod q\}
 \in\{\ell,\ell+1\}.
\]
Since $T/q=\ell+h/q$, we obtain
\[
 \left|
 \#\{1\le t\le T:t\equiv a\pmod q\}-\frac Tq
 \right|\le1.
\]
Dividing by $T$ proves
\[
 \left|
 \frac1T\sum_{t=1}^T J_{\bm d}(t)J_{\bm e}(t)
 -\E(J_{\bm d}J_{\bm e})
 \right|\le\frac1T.
\]
This estimate therefore holds in both cases, for every pair
$\bm d,\bm e\in\mathcal D_k$, uniformly in $b$ and the shifts.

Applying the triangle inequality to the expansion of $R^2$ and
using \eqref{S_DEF} and \eqref{eq:l1}, we obtain
\begin{equation}\label{eq:average-error}
 \begin{aligned}
 |S_1-\E R^2|
 &\le\frac1T\sum_{\bm d,\bm e}
                  |\lambda_{\bm d}|\,|\lambda_{\bm e}|\\
 &=\frac1T\left(\sum_{\bm d}|\lambda_{\bm d}|\right)^2
 \le\frac{D^3}{T}.
 \end{aligned}
\end{equation}

The same argument applies if any subset of the terms in
\eqref{eq:average-error} is retained. More precisely, for
$\mathcal J\subseteq\mathcal D_k$, set
\[
 R_{\mathcal J}(t)=\sum_{\bm d\in\mathcal J}
                                     \lambda_{\bm d}J_{\bm d}(t).
\]
Every retained pair has the same error bound, and
$\sum_{\bm d\in\mathcal J}|\lambda_{\bm d}|
 \le\sum_{\bm d\in\mathcal D_k}|\lambda_{\bm d}|$. Hence
\begin{equation}\label{R_J}
 \left|\frac1T\sum_{t=1}^TR_{\mathcal J}(t)^2
                          -\E R_{\mathcal J}^2\right|
 \le\frac{1}T
       \left(\sum_{\bm d\in\mathcal J}|\lambda_{\bm d}|\right)^2
 \le\frac{D^3}T.
\end{equation}
We shall use this bound for majorants of uncovered positions.
\end{proof}

\subsection{Correlations between the forms}

\begin{lemma}[A second-moment estimate]\label{lem:gram-estimate}
For arbitrary real coefficients $c_{\bm r}$ on $\mathcal D_j$,
where $0\le j\le k$, put
\[
 F(t)=\sum_{\bm r\in\mathcal D_j}c_{\bm r}
                           \prod_{i=1}^j\psi_{i,r_i}(t),
 \qquad
 I=\sum_{\bm r\in\mathcal D_j}
                           \frac{c_{\bm r}^2}
                                {\prod_{i=1}^j\varphi(r_i)}.
\]
Then
\begin{equation}\label{eq:gram}
 |\E F^2-I|\ll\frac{k\log D}{z}\,I.
\end{equation}
\end{lemma}

\begin{proof}
For $j=0$, $F$ is a constant, so $\E F^2=I$ and there is nothing
to prove. Assume that $1\le j\le k$.

For $\bm r\in\mathcal D_j$, write $n(\bm r)=r_1\cdots r_j$.
The coordinates are pairwise coprime, so this product is
squarefree and
$\varphi(n(\bm r))=\prod_i\varphi(r_i)$. Introduce the normalized
functions and coefficients
\[
 \Psi_{\bm r}(t)
   =\sqrt{\varphi(n(\bm r))}\prod_{i=1}^j\psi_{i,r_i}(t),
 \qquad
 b_{\bm r}=\frac{c_{\bm r}}{\sqrt{\varphi(n(\bm r))}}.
\]
Then $F=\sum_{\bm r}b_{\bm r}\Psi_{\bm r}$ and
$I=\sum_{\bm r}b_{\bm r}^2$. By the independence of the residue
coordinates $t\bmod p$ for distinct primes $p\mid P_*$,
established above, and by \eqref{eq:local-moments}, we obtain
\[
 \E\Psi_{\bm r}^2
 =\varphi(n(\bm r))
    \prod_{p\mid n(\bm r)}\frac1{p-1}=1.
\]
Thus $I$ is exactly the diagonal part of the expansion of
$\E F^2$. Let $M_{\bm r,\bm s}=\E(\Psi_{\bm r}\Psi_{\bm s})$ be the
matrix of normalized inner products. We obtain
\begin{equation}\label{E_F}
\E F^2-I = \sum_{\bm r \neq \bm s}M_{\bm r,\bm s} b_{\bm r}b_{\bm s}.
\end{equation}

If $n(\bm r)\ne n(\bm s)$,
some prime divides precisely one of these two squarefree
products. At that prime the local average is $\E\psi_{i,p}=0$,
so independence implies $M_{\bm r,\bm s}=0$.
It remains to consider tuples with the same product $n$.

Fix a tuple $\bm r\in\mathcal D_j$ and put
$n=n(\bm r)$. For each prime $p\mid n$, let
$\iota_{\bm r}(p)$ denote the unique index $i$ such that
$p\mid r_i$. This index is unique because the coordinates
of $\bm r$ are pairwise coprime.

Now consider another tuple $\bm s\in\mathcal D_j$ with
$n(\bm s)=n$, and define $\iota_{\bm s}(p)$ similarly.
The normalized functions can be written as
\[
 \Psi_{\bm r}(t)
 =
 \prod_{p\mid n}
 \bigl(\sqrt{p-1}\,\psi_{\iota_{\bm r}(p),p}(t)\bigr),
\]
and likewise for $\Psi_{\bm s}$. By independence across
distinct prime moduli,
\[
 M_{\bm r,\bm s}
 =
 \prod_{p\mid n}
 \left(
 (p-1)\,
 \E\bigl(
 \psi_{\iota_{\bm r}(p),p}
 \psi_{\iota_{\bm s}(p),p}
 \bigr)
 \right).
\]
For each fixed $p\mid n$, the local factor is $1$ if
$\iota_{\bm s}(p)=\iota_{\bm r}(p)$, and it is
$-1/(p-1)$ otherwise, by \eqref{eq:local-moments}. Define
\[
 h_p(\ell)=
 \begin{cases}
  1,&\ell=\iota_{\bm r}(p),\\[1mm]
  \dfrac1{p-1},&\ell\ne\iota_{\bm r}(p),
 \end{cases}
 \qquad 1\le\ell\le j.
\]
Taking absolute values in the preceding product gives
\[
 |M_{\bm r,\bm s}|
 =
 \prod_{p\mid n}h_p\bigl(\iota_{\bm s}(p)\bigr).
\]

We next describe precisely the tuples over which we sum.
For each prime $p\mid n$, choose an index
$f(p)\in\{1,\ldots,j\}$, and set
\[
 s_\ell=
 \prod_{\substack{p\mid n\\f(p)=\ell}}p,
 \qquad 1\le\ell\le j,
\]
where an empty product is understood to be $1$.
Every such choice produces a tuple in $\mathcal D_j$:
each $s_\ell$ divides $P_*$, the coordinates are pairwise
coprime, and their product is $n\le D$.
Conversely, every tuple $\bm s\in\mathcal D_j$ with
$n(\bm s)=n$ arises from exactly one such choice,
namely $f(p)=\iota_{\bm s}(p)$.

Thus the coordinate assigned to each prime can be chosen
freely, with no further restriction linking the choices
for different primes. Summing over all tuples with
product $n$ is therefore equivalent to summing over
all maps
\[
 f:\{p\in\PP:p\mid n\}\longrightarrow\{1,\ldots,j\}.
\]
By the distributive law for finite sums and products,
\[
 \begin{aligned}
 \sum_{\substack{\bm s\in\mathcal D_j\\n(\bm s)=n}}
       |M_{\bm r,\bm s}|
 &=
 \sum_{f:\{p\in\PP:p\mid n\}\to\{1,\ldots,j\}}
       \prod_{p\mid n}h_p(f(p))\\
 &=
 \prod_{p\mid n}
       \left(\sum_{\ell=1}^j h_p(\ell)\right)\\
 &=
 \prod_{p\mid n}
       \left(1+\frac{j-1}{p-1}\right).
 \end{aligned}
\]
Indeed, for each prime $p$, exactly one choice of $\ell$
agrees with $\iota_{\bm r}(p)$ and contributes $1$;
the remaining $j-1$ choices each contribute $1/(p-1)$.

This sum includes the diagonal entry
$|M_{\bm r,\bm r}|=1$, corresponding to the choice
$f(p)=\iota_{\bm r}(p)$ for every $p\mid n$.
Moreover, we have already shown that
$M_{\bm r,\bm s}=0$ whenever $n(\bm s)\ne n$.
Hence removing the diagonal entry gives the full
off-diagonal row sum:
\[
 \sum_{\substack{\bm s\in\mathcal D_j\\\bm s\ne\bm r}}
       |M_{\bm r,\bm s}|
 =
 \prod_{p\mid n}
       \left(1+\frac{j-1}{p-1}\right)-1.
\]
We note that this formula also holds for $n=1$, since both sides are $0$. Otherwise, $p>z$ and
$\omega(n)\le\log D/\log z$ (see \eqref{w(n).calc}) imply
\[
 \sum_{p\mid n}\frac{j-1}{p-1}
 \le\frac{2k\omega(n)}z
 \le\frac{2k\log D}{z\log z}
 \ll\frac{k\log D}{z}=O(x^{-4}).
\]
Here we used $k\le x$, $\log D=x/4$, and $z=x^6$.
Applying $1+h\le e^h$ and then $e^v-1\ll v$ for $0\le v\le1$,
we conclude that the absolute off-diagonal row sums satisfy
\[
 \rho:=\max_{\bm r\in\mathcal D_j}
       \sum_{\substack{\bm s\in\mathcal D_j\\\bm s\ne\bm r}}
       |M_{\bm r,\bm s}|
 \ll\frac{k\log D}{z}.
\]

Finally, $M$ is symmetric. Using \eqref{E_F} and $2|ab|\le a^2+b^2$, we obtain
\[
 \begin{aligned}
 |\E F^2-I|
 &\le\sum_{\bm r\ne\bm s}
         |M_{\bm r,\bm s}|\,|b_{\bm r}b_{\bm s}|\\
 &\le\frac12\sum_{\bm r\ne\bm s}
         |M_{\bm r,\bm s}|(b_{\bm r}^2+b_{\bm s}^2)\\
 &=\sum_{\bm r}b_{\bm r}^2
       \sum_{\bm s\ne\bm r}|M_{\bm r,\bm s}|\\
 &\le\rho\sum_{\bm r}b_{\bm r}^2
 \ll\frac{k\log D}{z}\,I.
 \end{aligned}
\]
This proves \eqref{eq:gram} for arbitrary real coefficients
$c_{\bm r}$.
\end{proof}

\subsection{The diagonal mass and total weight}

\begin{lemma}[The total weight]\label{lem:total-weight}
With $S_1$ defined by \eqref{S_DEF}, we have
\[
 S_1=(1+o(1))A^k>0
\]
uniformly for $1\le k\le\eta x$, provided the absolute
constant $\eta$ is sufficiently small.
\end{lemma}

\begin{proof}
For the coefficients of $R$, put
\[
 I_R=\sum_{\bm r\in\mathcal D_k}
                 \prod_{i=1}^k\frac{a(r_i)^2}{\varphi(r_i)}.
\]
Our aim is to prove that
\[
 I_R=(1+o(1))A^k.
\]
We first compute the larger sum obtained by
removing both the pairwise coprimality condition and
the restriction $r_1\cdots r_k\le D$.

In this unrestricted sum, each coordinate $r_i$ ranges
freely over all divisors of $P_*$. Since the $i$th factor
depends only on $r_i$, the sum factors as
\[
 \begin{aligned}
 \sum_{r_1\mid P_*}\cdots\sum_{r_k\mid P_*}
       \prod_{i=1}^k\frac{a(r_i)^2}{\varphi(r_i)}
 &=
 \prod_{i=1}^k
       \left(
       \sum_{r_i\mid P_*}
       \frac{a(r_i)^2}{\varphi(r_i)}
       \right)\\
 &=A^k,
 \end{aligned}
\]
where the last equality follows from the definition of $A$.

The sum defining $I_R$ contains only the tuples belonging
to $\mathcal D_k$. Thus $I_R$ is obtained from the
unrestricted sum by discarding some nonnegative terms.
In particular,
\[
 \begin{aligned}
 0\le A^k-I_R
 &=
 \sum_{\substack{r_i\mid P_*\ (1\le i\le k)\\
                  (r_1,\ldots,r_k)\notin\mathcal D_k}}
       \prod_{i=1}^k\frac{a(r_i)^2}{\varphi(r_i)}.
 \end{aligned}
\]

A tuple is discarded if its product exceeds $D$, or if
two of its coordinates have a common prime divisor.
We estimate separately the total contribution of each
of these two types of excluded tuples.
A tuple violating both conditions may be counted twice;
this is harmless because we seek an upper bound and
all summands are nonnegative.

It therefore suffices to show that each of these two
contributions is $o(A^k)$. This will imply that
$A^k-I_R=o(A^k)$, and hence that $I_R/A^k\to1$.

First, for $r_1\cdots r_k>D$ and $\beta>0$, we have
$1\le(r_1\cdots r_k/D)^\beta$. Thus the mass beyond the cutoff
is at most
\[
 \begin{aligned}
 &\sum_{\substack{r_i\mid P_*\ (1\le i\le k)\\
                   r_1\cdots r_k>D}}
              \prod_{i=1}^k\frac{a(r_i)^2}{\varphi(r_i)}\\
 &\qquad\le D^{-\beta}
       \sum_{r_1\mid P_*}\cdots\sum_{r_k\mid P_*}
              \prod_{i=1}^k\frac{a(r_i)^2r_i^\beta}{\varphi(r_i)}
       =D^{-\beta}A_\beta^k.
 \end{aligned}
\]
Since $A\ge1$ and $A_\beta\le A+C_1\beta$, we have
\[
 A_\beta^k
 \le A^k\left(1+\frac{C_1\beta}{A}\right)^k
 \le A^k e^{C_1k\beta}.
\]
We now choose the constant in the lemma so that
\[
 0<\eta\le\min\left(1,\frac1{8C_1}\right).
\]
This ensures
$C_1k\le C_1\eta x\le x/8=\tfrac12\log D$.
Because $u=\beta\log D$, it follows that $C_{1}k\beta \le u/2$ and
\begin{equation}\label{eq:tail}
 D^{-\beta}A_\beta^k
 \le A^k\exp(-u+C_1k\beta)
 \le A^k e^{-u/2}.
\end{equation}

Second, we estimate the contribution of tuples whose
coordinates are not pairwise coprime. If $k=1$, there
are no pairs of distinct coordinates, so this
contribution is zero. We therefore assume that $k\ge2$.

If a tuple $(r_1,\ldots,r_k)$ fails the pairwise
coprimality condition, then $(r_i,r_j)>1$ for some
indices $i<j$. Hence there is a prime $p$ such that
\[
 p\mid r_i,\qquad p\mid r_j.
\]
Since every coordinate divides $P_*$, this prime
also divides $P_*$.

For each prime $p\mid P_*$, put
\[
 B_p=
 \sum_{\substack{d\mid P_*\\p\mid d}}
       \frac{a(d)^2}{\varphi(d)}.
\]
Fix a pair $i<j$ and a prime $p\mid P_*$. We have
\[
 \begin{aligned}
 &\sum_{\substack{r_\ell\mid P_*\ (1\le\ell\le k)\\
                  p\mid r_i,\ p\mid r_j}}
       \prod_{\ell=1}^k
       \frac{a(r_\ell)^2}{\varphi(r_\ell)}
 \\
 &\qquad=
 \left(
 \sum_{\substack{r_i\mid P_*\\p\mid r_i}}
       \frac{a(r_i)^2}{\varphi(r_i)}
 \right)
 \left(
 \sum_{\substack{r_j\mid P_*\\p\mid r_j}}
       \frac{a(r_j)^2}{\varphi(r_j)}
 \right)
 \prod_{\substack{1\le\ell\le k\\\ell\ne i,j}}
 \left(
 \sum_{r_\ell\mid P_*}
       \frac{a(r_\ell)^2}{\varphi(r_\ell)}
 \right)
 \\
 &\qquad=B_p^2A^{k-2}.
 \end{aligned}
\]

Every tuple failing pairwise coprimality is included
for at least one pair $(i,j)$ and at least one prime
$p\mid P_*$. A tuple may be included several times,
but all weights are nonnegative, so summing over
all pairs and primes gives an upper bound:
\[
 \begin{aligned}
 &\sum_{\substack{r_\ell\mid P_*\ (1\le\ell\le k)\\
          (r_i,r_j)>1\ \text{for some }i<j}}
       \prod_{\ell=1}^k
       \frac{a(r_\ell)^2}{\varphi(r_\ell)}\le
       \sum_{1\le i < j \le k} \sum_{p|P_{*}}
       \sum_{\substack{r_\ell\mid P_*\ (1\le\ell\le k)\\
                  p\mid r_i,\ p\mid r_j}}
       \prod_{\ell=1}^k
       \frac{a(r_\ell)^2}{\varphi(r_\ell)}
 \\
 &\qquad\le
 \sum_{1\le i<j\le k}\sum_{p\mid P_*}
       A^{k-2}B_p^2
 =
 \binom{k}{2}A^{k-2}\sum_{p\mid P_*}B_p^2.
 \end{aligned}
\]

It remains to bound $B_p$. By
\eqref{eq:basic-coefficients}, we have $|a(d)|\le1$,
and therefore $a(d)^2\le|a(d)|$. Applying
\eqref{eq:prime-coefficients}, we obtain
\[
 B_p
 \le
 \sum_{\substack{d\mid P_*\\p\mid d}}
       \frac{|a(d)|}{\varphi(d)}
 \ll\frac1p.
\]
All prime divisors of $P_*$ exceed $z$. Hence
\[
 \sum_{p\mid P_*}B_p^2
 \ll
 \sum_{p\mid P_*}\frac1{p^2}
 \le
 \sum_{\substack{n\in\mathbb N\\n>z}}\frac1{n^2}
 \ll\frac1z.
\]
For completeness, the last estimate follows from
\[
 \sum_{\substack{n\in\mathbb N\\n>z}}\frac1{n^2}
 \le
 \int_{\lfloor z\rfloor}^{\infty}\frac{dt}{t^2}
 =
 \frac1{\lfloor z\rfloor}
 \le\frac2z
 \qquad(z\ge2).
\]
Substituting these bounds and using $A\ge1$, we conclude
that
\[
 \sum_{\substack{r_\ell\mid P_*\ (1\le\ell\le k)\\
          (r_i,r_j)>1\ \text{for some }i<j}}
       \prod_{\ell=1}^k
       \frac{a(r_\ell)^2}{\varphi(r_\ell)}
 \ll
 \frac{k^2}{z}A^{k-2}
 \le
 \frac{k^2}{z}A^k.
\]
Thus, the total weight of all tuples with
$(r_i,r_j)>1$ for some $i<j$ is $O(k^2A^k/z)$.

Combining the tail bound~\eqref{eq:tail} with the preceding
coprimality bound, we have
\[
 0\le A^k-I_R
 \ll A^k\left(e^{-u/2}+\frac{k^2}{z}\right).
\]
Here $u=\tfrac14(\log x)^5\to\infty$ and
$k^2/z\le x^2/x^6=x^{-4}\to0$. Hence
\[
 I_R=A^k\left(1+O\!\left(e^{-u/2}+\frac{k^2}{z}\right)\right)
 =(1+o(1))A^k.
\]
We now apply \eqref{eq:gram} to the function $R$.
In the general definition of $F$, take
\[
 j=k,
 \qquad
 c_{\bm r}=\prod_{i=1}^k a(r_i)
 \quad(\bm r\in\mathcal D_k).
\]
With this choice,
\[
 \begin{aligned}
 F(t)
 &=
 \sum_{\bm r\in\mathcal D_k}
 c_{\bm r}\prod_{i=1}^k\psi_{i,r_i}(t)\\
 &=
 \sum_{\bm r\in\mathcal D_k}
 \prod_{i=1}^k
       \bigl(a(r_i)\psi_{i,r_i}(t)\bigr)
 =R(t).
 \end{aligned}
\]
Moreover, the corresponding diagonal sum is precisely
$I_R$, since
\[
 \begin{aligned}
 I
 &=
 \sum_{\bm r\in\mathcal D_k}
 \frac{c_{\bm r}^{\,2}}
      {\prod_{i=1}^k\varphi(r_i)}\\
 &=
 \sum_{\bm r\in\mathcal D_k}
 \frac{\left(\prod_{i=1}^k a(r_i)\right)^2}
      {\prod_{i=1}^k\varphi(r_i)}\\
 &=
 \sum_{\bm r\in\mathcal D_k}
 \prod_{i=1}^k\frac{a(r_i)^2}{\varphi(r_i)}
 =I_R.
 \end{aligned}
\]
Thus \eqref{eq:gram}, with $F=R$ and $I=I_R$, gives
\[
 \left|\E R^2-I_R\right|
 \ll\frac{k\log D}{z}\,I_R.
\]
Since $I_R\le A^k$, it follows that
\[
 \E R^2
 =
 I_R+O\!\left(\frac{k\log D}{z}A^k\right).
\]

We have already proved that
\[
 I_R
 =
 A^k+
 O\!\left(
 A^k\left(e^{-u/2}+\frac{k^2}{z}\right)
 \right).
\]
Substituting this estimate into the preceding identity
and combining the error terms, we obtain
\[
 \begin{aligned}
 \E R^2
 &=
 A^k+
 O\!\left(
 A^k\left(
 e^{-u/2}+\frac{k^2}{z}+\frac{k\log D}{z}
 \right)
 \right)\\
 &=
 A^k\left(
 1+O\!\left(
 e^{-u/2}+\frac{k^2}{z}+\frac{k\log D}{z}
 \right)
 \right).
 \end{aligned}
\]
Finally, $T=\lfloor e^x\rfloor\ge e^x/2$ for large $x$, so
\[
 \frac{D^3}{T}\le2e^{3x/4-x}=2e^{-x/4}.
\]
Since $A^k\ge1$, the absolute error in~\eqref{eq:average-error}
is also a relative error of size at most $2e^{-x/4}$ after
division by $A^k$. Therefore
\[
 \frac{S_1}{A^k}
 =1+O\!\left(e^{-u/2}+\frac{k^2}{z}
                  +\frac{k\log D}{z}+e^{-x/4}\right)
 =1+o(1).
\]
All terms tend to zero uniformly for $k\le\eta x$. In particular,
\begin{equation}\label{eq:S1}
 S_1=(1+o(1))A^k>0
\end{equation}
for sufficiently large $x$.
\end{proof}

\subsection{Controlling uncovered forms}

\begin{lemma}[The contribution of uncovered forms]\label{lem:uncovered-forms}
For $S_2$ defined by \eqref{S_DEF}, we have
\[
 S_2\ll kA^k\left(e^{-u}+\frac{(\log D)^2}{z^2}\right)
           +\frac{kD^3}{T}.
\]
The estimate is uniform in the admissible shifts and in $b$.
\end{lemma}

\begin{proof}
We first estimate the contribution of $L_k$; the same argument
will apply to each other form after relabelling. For $v\mid P_*$
with $v\le D$, we define
\[
 \varepsilon(v)=
   \sum_{\substack{d\mid P_*\\d\le D/v\\(d,v)=1}}
                    \frac{a(d)}{\varphi(d)},
\]
and put
\[
 R_k(t)=\sum_{\bm r\in\mathcal D_{k-1}}
       \varepsilon(r_1\cdots r_{k-1})
       \prod_{i=1}^{k-1}a(r_i)\psi_{i,r_i}(t).
\]
To explain this definition, fix the first $k-1$ coordinates of
a tuple in $\mathcal D_k$ and write $v=r_1\cdots r_{k-1}$.
The admissible last coordinates are exactly those divisors
$d\mid P_*$ satisfying $(d,v)=1$ and $d\le D/v$. Grouping the
terms of $R$ according to their first $k-1$ coordinates gives
the exact identity
\[
 R(t)=\sum_{\bm r\in\mathcal D_{k-1}}
       \left(\prod_{i=1}^{k-1}a(r_i)\psi_{i,r_i}(t)\right)
       \sum_{\substack{d\mid P_*\\d\le D/v\\(d,v)=1}}
                         a(d)\psi_{k,d}(t),
 \qquad v=r_1\cdots r_{k-1}.
\]
If $(L_k(t),P_*)=1$, then $\ind_{p\mid L_k(t)}=0$ for every
$p\mid P_*$, and consequently
\[
 \psi_{k,d}(t)
 =\prod_{p\mid d}\frac1{p-1}
 =\frac1{\varphi(d)}\qquad(d\mid P_*).
\]
The inner sum is then $\varepsilon(v)$, so $R(t)=R_k(t)$.
Since $w(t) = R(t)^2$, we obtain
\begin{equation}\label{eq:uncovered}
 w(t)\ind_{(L_k(t),P_*)=1}\le R_k(t)^2.
\end{equation}
Indeed, when $(L_k (t), P_{*})=1$, both sides are equal, and when $(L_{k}(t), P_{*})>1$ the left-hand side is zero.
When $k=1$, the set $\mathcal D_0$ contains only the empty tuple,
and these formulas mean $v=1$ and $R_1(t)=\varepsilon(1)$.

We also need to compare the two averages of $R_k^2$. Algebraically,
$R_k$ is obtained from $R$ by setting every prime divisibility
indicator for the $k$th form equal to zero. Indeed, this replaces
$\psi_{k,d}(t)$ by $1/\varphi(d)$ in the product expansion.
In the expansion with coefficients $\lambda_{\bm d}$, this
operation removes exactly the terms with $d_k>1$, and leaves
those with $d_k=1$. Hence (see \eqref{R_VIEW})
\[
 R_k(t)=\sum_{(d_1,\ldots,d_{k-1})\in\mathcal D_{k-1}}
       \lambda_{(d_1,\ldots,d_{k-1},1)}
       \prod_{i=1}^{k-1}\ind_{d_i\mid L_i(t)}.
\]
Since $(d_1,\ldots, d_{k-1}, 1)\in \mathcal D_k$ for every $(d_1,\ldots, d_{k-1})\in \mathcal D_{k-1}$, it follows from \eqref{R_J} that
\begin{equation}\label{eq:collapsed-average}
 \left|\frac1T\sum_{t=1}^T R_k(t)^2-\E R_k^2\right|
 \le\frac1T\left(\sum_{\bm d\in\mathcal D_{k-1}}
                     |\lambda_{(\bm d,1)}|\right)^2
 \le\frac{D^3}{T}.
\end{equation}

\medskip\noindent\textit{Estimating the remaining coefficient sum.}
The full sum $\sum_{d\mid P_*}a(d)/\varphi(d)$ is zero by
\eqref{eq:basic-coefficients}. Since $v\le D$, the divisor $d=1$
occurs in the sum defining $\varepsilon(v)$. Every omitted divisor
is therefore greater than $1$, and its coefficient $a(d)$ is
negative. Subtracting the omitted terms from the zero full sum
gives the exact identity
\[
 \varepsilon(v)
 =-\sum_{\substack{d\mid P_*\\d>D/v\ \text{or }(d,v)>1}}
                 \frac{a(d)}{\varphi(d)}
 =\sum_{\substack{d\mid P_*\\d>D/v\ \text{or }(d,v)>1}}
                 \frac{|a(d)|}{\varphi(d)}\ge0.
\]
The condition $(d,v)>1$ means that some prime dividing $v$
also divides $d$. Taking the sum over these two possible reasons
for exclusion, with any multiple counting increasing the bound,
we obtain
\[
 \varepsilon(v)
 \le\sum_{\substack{d\mid P_*\\d>D/v}}
                    \frac{|a(d)|}{\varphi(d)}
    +\sum_{p\mid v}\sum_{\substack{d\mid P_*\\p\mid d}}
                    \frac{|a(d)|}{\varphi(d)}.
\]
For the first sum, $d>D/v$ implies $1\le(vd/D)^\beta$, so
\[
 \begin{aligned}
 \sum_{\substack{d\mid P_*\\d>D/v}}
          \frac{|a(d)|}{\varphi(d)}
 &\le\left(\frac vD\right)^\beta
       \sum_{\substack{d\mid P_*\\d>D/v}}
          \frac{|a(d)|d^\beta}{\varphi(d)}\\
 &\le\left(\frac vD\right)^\beta
       \sum_{\substack{d\mid P_*\\d>1}}
          \frac{|a(d)|d^\beta}{\varphi(d)}
 \ll\left(\frac vD\right)^\beta,
 \end{aligned}
\]
by~\eqref{eq:tilt}. For the second sum, use
\eqref{eq:prime-coefficients}, $p>z$, and
$\omega(v)\le\log v/\log z$ to get
\[
 \begin{aligned}
 \sum_{p\mid v}\sum_{\substack{d\mid P_*\\p\mid d}}
          \frac{|a(d)|}{\varphi(d)}
 &\ll\sum_{p\mid v}\frac1p
 \le\frac{\omega(v)}z\\
 &\le\frac{\log v}{z\log z}
 \le\frac{\log D}{z\log z}
 \le\frac{\log D}{z}.
 \end{aligned}
\]
We note that for $v=1$ both prime sums are empty and these inequalities still
hold. Combining the estimates proves
\[
 0\le\varepsilon(v)
 \ll(v/D)^\beta+\frac{\log D}{z}.
\]
Squaring and applying $(a+b)^2\le2a^2+2b^2$, we also have
\begin{equation}\label{est:eps.2}
 \varepsilon(v)^2
 \ll D^{-2\beta}v^{2\beta}
      +\frac{(\log D)^2}{z^2}.
\end{equation}

\medskip\noindent\textit{The second moment of the collapsed weight.}
The diagonal sum for $R_k$ is
\[
 I_k=\sum_{\bm r\in\mathcal D_{k-1}}
       \varepsilon(v)^2
       \prod_{i=1}^{k-1}\frac{a(r_i)^2}{\varphi(r_i)},
 \qquad v=r_1\cdots r_{k-1}.
\]
For every tuple $\bm r\in\mathcal D_{k-1}$, the product
$v=r_1\cdots r_{k-1}$ satisfies $v\mid P_*$ and $v\le D$.
By \eqref{est:eps.2}, we therefore have
\[
 \varepsilon(v)^2
 \le C\left(
 D^{-2\beta}v^{2\beta}
 +\frac{(\log D)^2}{z^2}
 \right)
\]
with an absolute constant $C>0$, uniformly over these
tuples. Moreover,
\[
 v^{2\beta}
 =(r_1\cdots r_{k-1})^{2\beta}
 =\prod_{i=1}^{k-1}r_i^{2\beta}.
\]
Multiplying the bound for $\varepsilon(v)^2$ by the
nonnegative weight
$\prod_{i=1}^{k-1}a(r_i)^2/\varphi(r_i)$
and summing over $\mathcal D_{k-1}$, we obtain
\[
 \begin{aligned}
 I_k
 &\le
 CD^{-2\beta}
 \sum_{\bm r\in\mathcal D_{k-1}}
       \prod_{i=1}^{k-1}
       \frac{a(r_i)^2r_i^{2\beta}}{\varphi(r_i)}
 \\
 &\quad+
 C\frac{(\log D)^2}{z^2}
 \sum_{\bm r\in\mathcal D_{k-1}}
       \prod_{i=1}^{k-1}
       \frac{a(r_i)^2}{\varphi(r_i)}.
 \end{aligned}
\]

We now enlarge the domain of each sum. Membership in
$\mathcal D_{k-1}$ imposes the conditions
\[
 r_i\mid P_*,
 \qquad
 (r_i,r_j)=1\quad(i\ne j),
 \qquad
 r_1\cdots r_{k-1}\le D.
\]
We retain only the conditions $r_i\mid P_*$ and remove
the coprimality and product restrictions.
Since every summand is nonnegative, this can only
increase the sums.

For $\gamma\in\{0,2\beta\}$, it follows that
\[
 \begin{aligned}
 &\sum_{\bm r\in\mathcal D_{k-1}}
       \prod_{i=1}^{k-1}
       \frac{a(r_i)^2r_i^\gamma}{\varphi(r_i)}
 \\
 &\quad\le
 \sum_{r_1\mid P_*}\cdots\sum_{r_{k-1}\mid P_*}
       \prod_{i=1}^{k-1}
       \frac{a(r_i)^2r_i^\gamma}{\varphi(r_i)}
 \\
 &\quad=
 \prod_{i=1}^{k-1}
 \left(
 \sum_{r_i\mid P_*}
       \frac{a(r_i)^2r_i^\gamma}{\varphi(r_i)}
 \right)
 =A_\gamma^{k-1},
 \end{aligned}
\]
by the definition of $A_\gamma$.

Taking $\gamma=2\beta$ in the first sum and $\gamma=0$
in the second, and recalling that $A_0=A$, we conclude
that
\[
 I_k
 \ll
 D^{-2\beta}A_{2\beta}^{k-1}
 +\frac{(\log D)^2}{z^2}A^{k-1}.
\]

For completeness, let us consider the case $k=1$ separately.
Then $\mathcal D_{k-1}=\mathcal D_0$ contains exactly one
element, namely the empty tuple. Its coordinate product
is $v=1$, and the empty product of weights also equals $1$.
Thus the sum defining $I_1$ has a single term, and
\[
 I_1=\varepsilon(1)^2.
\]
The estimate \eqref{est:eps.2} is valid at $v=1$, since $1\mid P_*$ and $1\le D$.
Substituting $v=1$ therefore gives
\[
 I_1=\varepsilon(1)^2
 \ll D^{-2\beta}+\frac{(\log D)^2}{z^2}.
\]
This is precisely the claimed bound for $I_k$ when
$k=1$, because
\[
 A_{2\beta}^{k-1}=A_{2\beta}^{0}=1,
 \qquad
 A^{k-1}=A^0=1.
\]
Hence the estimate for $I_k$ holds also in this case.

Applying~\eqref{eq:gram} with $j=k-1$ to the coefficients
$\varepsilon(v)\prod_{i<k}a(r_i)$ yields
\[
 \E R_k^2
 \le\left(1+C_2\frac{k\log D}{z}\right)I_k
 \ll D^{-2\beta}A_{2\beta}^{k-1}
          +\frac{(\log D)^2}{z^2}A^{k-1}.
\]
Here $C_2$ is an absolute constant, and the multiplier is bounded
because $k\log D/z=o(1)$.

We now estimate the first term.
Taking $\gamma=2\beta$ in \eqref{eq:tilt}, we obtain
\[
 A_{2\beta}\le A+2C_1\beta.
\]
Since $k\ge1$ and $A\ge1$, it follows that
\[
 \begin{aligned}
 A_{2\beta}^{k-1}
 &\le (A+2C_1\beta)^{k-1}\\
 &=A^{k-1}
   \left(1+\frac{2C_1\beta}{A}\right)^{k-1}\\
 &\le A^{k-1}
   \exp\!\left(\frac{2C_1(k-1)\beta}{A}\right)\\
 &\le A^{k-1}
   \exp\!\left(2C_1(k-1)\beta\right).
 \end{aligned}
\]

By the definition $u=\beta\log D$, we have
\[
 D^{-2\beta}
 =\exp(-2\beta\log D)
 =e^{-2u}.
\]
Multiplying the preceding estimate by this factor gives
\[
 \begin{aligned}
 D^{-2\beta}A_{2\beta}^{k-1}
 &\le
 A^{k-1}\exp\!\left(-2u+2C_1(k-1)\beta\right)\\
 &\le
 A^k\exp\!\left(-2u+2C_1k\beta\right).
 \end{aligned}
\]

It remains to control the positive term in the exponent.
The constant $\eta$ was chosen so that
$\eta\le1/(8C_1)$. Since $k\le\eta x$ and
$\log D=x/4$, we have
\[
 2C_1k
 \le2C_1\eta x
 \le\frac{x}{4}
 =\log D.
\]
Multiplying by $\beta>0$ yields
\[
 2C_1k\beta
 \le\beta\log D
 =u.
\]
Therefore
\[
 -2u+2C_1k\beta\le-u,
\]
and consequently
\[
 D^{-2\beta}A_{2\beta}^{k-1}
 \le A^k e^{-u}.
\]
The second term is at most $(\log D)^2A^k/z^2$, again because
$A\ge1$. We have therefore proved
\[
 \E R_k^2
 \ll A^k\left(e^{-u}+\frac{(\log D)^2}{z^2}\right).
\]

\medskip\noindent\textit{Bounding the number of uncovered forms.}
Combining~\eqref{eq:uncovered} and~\eqref{eq:collapsed-average},
we find
\[
 \frac1T\sum_{t=1}^T
       w(t)\ind_{(L_k(t),P_*)=1}
 \le\E R_k^2+\frac{D^3}{T}
 \ll A^k\left(e^{-u}+\frac{(\log D)^2}{z^2}\right)
       +\frac{D^3}{T}.
\]
The tuple set, the coefficient bounds, and the local moment
estimates are invariant under a permutation of the forms.
Repeating the same argument with each $L_i$ as the last form
and adding the $k$ estimates gives
\[
 \begin{aligned}
 S_2
 &=\sum_{i=1}^k\frac1T\sum_{t=1}^T
                    w(t)\ind_{(L_i(t),P_*)=1}\\
 &\ll kA^k\left(e^{-u}+\frac{(\log D)^2}{z^2}\right)
           +\frac{kD^3}{T}.
 \end{aligned}
\]
\end{proof}

\begin{proof}[Completion of the proof of Lemma~\ref{lem:translate}]
By Lemma~\ref{lem:total-weight}, $S_1\ge A^k/2$ for sufficiently large $x$.
Using Lemma~\ref{lem:uncovered-forms} and dividing by this lower
bound, together with $A^k\ge1$ and the parameter
choices in~\eqref{eq:parameters}, which give
$u=(\log x)^5/4$, $\log D=x/4$, $z=x^6$, and
$D^3/T\le2e^{-x/4}$, we obtain
\[
 \begin{aligned}
 \frac{S_2}{S_1}
 &\ll k e^{-u}+\frac{k(\log D)^2}{z^2}
                    +\frac{kD^3}{TA^k}\\
 &\ll k e^{-(\log x)^5/4}
          +\frac{kx^2}{x^{12}}+k e^{-x/4}\\
 &\ll x e^{-(\log x)^5/4}+x^{-9}+x e^{-x/4}
 =o(1).
 \end{aligned}
\]
Thus $S_2<S_1$ for sufficiently large $x$.

If every $1\le t\le T$ had an uncovered form, its number of
uncovered forms would be at least $1$, and the nonnegative
weights would give $S_2\ge S_1$, a contradiction. Hence some
$t\in\{1,\ldots,T\}$ satisfies $(L_i(t),P_*)>1$ for every $i$.
Each such greatest common divisor has a prime factor, and every
prime factor of $P_*$ lies in $(z,y]$. Since $T\le e^x$, this is
exactly the divisibility statement of the lemma. Lemma \ref{lem:translate} is proved.
\end{proof}

\section{Proof of the block theorem}\label{sec:block-proof}

\begin{proof}[Proof of Theorem~\ref{thm:block}]
Let $\eta$ be the constant in Lemma~\ref{lem:translate}.
By decreasing it if necessary, we may assume that $0<\eta\le1$.
Fix an absolute constant $c_*>0$ such that
\begin{equation}\label{eq:c-choice}
 0<c_*\le1,\qquad 3e^{-\gamma_E}c_*<\frac\eta2.
\end{equation}
We take $X$ sufficiently large in terms of these absolute constants.
All lower bounds on $X$ imposed below will be independent of $\cB$.

\medskip\noindent\textit{The preliminary sieve.}
Put
\begin{equation}\label{eq:block-parameters}
 x=\frac13\log X,\qquad Q=\prod_{p\le x}p,\qquad
 M=B(X),\qquad F(X)=\log X\log\log X.
\end{equation}
In particular, $X=e^{3x}$. Define $H$ by \eqref{eq:H}.
We may assume that $H\ge1$, since otherwise there is nothing to prove.
Then
\[
 1\le H\le\frac{c_*F(X)}{M+1}.
\]
Consequently,
\begin{equation}\label{eq:size-bounds}
 M+1\le c_*F(X),\qquad
 H\le c_*F(X),\qquad
 HM\le H(M+1)\le c_*F(X).
\end{equation}

Only elements $a\in\cB\cap[0,X]$ can occur in a representation
$p+a$ of an integer at most $X$. Consider the set of shifts
\begin{equation}\label{eq:initial-shifts}
 S_0=\{j-a:1\le j\le H,\ a\in\cB\cap[0,X]\}.
\end{equation}
There are $HM$ pairs $(j,a)$ in \eqref{eq:initial-shifts}, so
\[
 |S_0|\le HM.
\]
Repeated shifts are included only once, as required by
Lemma~\ref{lem:translate}. When $M=0$, the set $S_0$ is empty and
the sums below vanish.

For each residue class $r\bmod Q$, let
\[
 K(r)=\#\{s\in S_0:(r+s,Q)=1\}.
\]
For each fixed $s$, exactly $\varphi(Q)$ residue classes $r\bmod Q$
satisfy $(r+s,Q)=1$, since translation by $s$ permutes the classes.
Thus
\begin{equation}\label{eq:average}
 \begin{aligned}
 \frac1Q\sum_{r\bmod Q}K(r)
 &=\frac1Q\sum_{r\bmod Q}\sum_{s\in S_0}
       \ind_{(r+s,Q)=1}\\
 &=\frac1Q\sum_{s\in S_0}\sum_{r\bmod Q}
       \ind_{(r+s,Q)=1}\\
 &=|S_0|\frac{\varphi(Q)}Q.
 \end{aligned}
\end{equation}
Choose $r\in\{0,\ldots,Q-1\}$ for which $K(r)$ does not exceed
this average, and put
\[
 S=\{s\in S_0:(r+s,Q)=1\}.
\]
It follows from \eqref{eq:average} and \eqref{eq:size-bounds} that
\begin{equation}\label{eq:survivors}
 |S|\le |S_0|\frac{\varphi(Q)}Q
 \le HM\frac{\varphi(Q)}Q
 \le c_*F(X)\frac{\varphi(Q)}Q.
\end{equation}

For every $s\in S_0\setminus S$, some prime $p\mid Q$ divides
$r+s$. It follows that $p\mid N+s$ whenever $N\equiv r\pmod Q$,
with $p\le x$. We shall use Lemma~\ref{lem:translate} to cover the
remaining shifts in $S$.

By Mertens' product formula,
\[
 \frac{\varphi(Q)}Q
 =\prod_{p\le x}\left(1-\frac1p\right)
 =\frac{e^{-\gamma_E}+o(1)}{\log x}.
\]
By \eqref{eq:block-parameters}, we have
\[
 F(X)=3x\log(3x)=3x(\log x+\log3).
\]
Substituting these two expressions into \eqref{eq:survivors},
we find
\[
 \begin{aligned}
 |S|
 &\le 3c_*x(\log x+\log3)
       \frac{e^{-\gamma_E}+o(1)}{\log x}\\
 &=\bigl(3c_*e^{-\gamma_E}+o(1)\bigr)x.
 \end{aligned}
\]
Our choice \eqref{eq:c-choice} therefore ensures that
\begin{equation}\label{eq:shift-count}
 |S|\le\eta x
\end{equation}
for all sufficiently large $X$, uniformly in $\cB$, since
\eqref{eq:size-bounds} has removed the dependence on $M$.

If $s=j-a\in S_0$, then $1\le j\le H$ and $0\le a\le X$,
so $1-X\le s\le H$. Moreover,
\[
 \frac{H}{X}\le c_*\frac{\log X\log\log X}{X}\longrightarrow0
\]
by \eqref{eq:size-bounds}. We may consequently assume that $H\le X$,
which yields
\begin{equation}\label{eq:shift-size}
 |s|\le X=e^{3x}\qquad(s\in S_0).
\end{equation}
If $S\ne\varnothing$, the distinct shifts in $S$ therefore satisfy
the hypotheses of Lemma~\ref{lem:translate}, by
\eqref{eq:shift-count} and \eqref{eq:shift-size}.

\medskip\noindent\textit{An interval disjoint from $\cB$.}
List the elements of $\cB$ strictly between $X/2$ and $X$ in increasing
order as $a_1<\cdots<a_q$, where $0\le q\le M$.
Set $a_0=X/2$ and $a_{q+1}=X$. The $q+1$ open intervals
$(a_i,a_{i+1})$, for $0\le i\le q$, contain no element of $\cB$,
and their lengths satisfy
\[
 \sum_{i=0}^{q}(a_{i+1}-a_i)=X/2.
\]
At least one of these lengths is at least $X/(2(q+1))$.
Choose one such interval and denote it by $(u,v)$. Then
\begin{equation}\label{eq:empty-gap}
 X/2\le u<v\le X,\qquad
 (u,v)\cap\cB=\varnothing,\qquad
 L:=v-u\ge\frac{X}{2(M+1)}.
\end{equation}
If $q=0$, we take $(u,v)=(X/2,X)$. Either endpoint may belong to $\cB$.

Combining \eqref{eq:empty-gap} with the first inequality in
\eqref{eq:size-bounds}, we obtain the uniform estimate
\begin{equation}\label{eq:L-lower}
 L\ge\frac{X}{2c_*F(X)}.
\end{equation}
We next compare \eqref{eq:L-lower} with the range of translates
in Lemma~\ref{lem:translate}. Put
\begin{equation}\label{eq:y}
 y=\exp\!\left(\frac{x}{(\log x)^5}\right).
\end{equation}
The prime number theorem, in its equivalent form
$\sum_{p\le x}\log p=x+o(x)$, gives
\[
 \log Q=x+o(x).
\]
Since $\log X=3x$, this implies
\begin{equation}\label{eq:translation-scale}
 Qe^x=\exp(2x+o(x))=X^{2/3+o(1)}.
\end{equation}
Also, the definition \eqref{eq:y} gives
\[
 \frac{\log y}{\log X}=\frac{1}{3(\log x)^5}\longrightarrow0,
 \qquad
 \frac{\log y}{\log x}=\frac{x}{(\log x)^6}\longrightarrow\infty.
\]
In particular, $y=X^{o(1)}$ and $y>x$ for all sufficiently large $X$.

Using \eqref{eq:L-lower}, \eqref{eq:translation-scale}, and
$Q\le Qe^x$, we have
\[
 \begin{aligned}
 \frac{Q+Qe^x+H}{L}
 &\le\frac{2c_*F(X)}X\bigl(2Qe^x+c_*F(X)\bigr)\\
 &\ll F(X)X^{-1/3+o(1)}+\frac{F(X)^2}{X}
 \longrightarrow0.
 \end{aligned}
\]
Similarly,
\[
 \frac yL\le 2c_*F(X)X^{-1+o(1)}\longrightarrow0.
\]
Here we have used $F(X)=\log X\log\log X=X^{o(1)}$.
Consequently, for all sufficiently large $X$,
\begin{equation}\label{eq:room}
 Q+Qe^x+H<\frac L3,\qquad x<y<\frac L3.
\end{equation}

\medskip\noindent\textit{The final translate.}
Choose
\[
 b=r+Q\left\lceil\frac{u+L/3-r}{Q}\right\rceil.
\]
Since $\xi \le \lceil \xi \rceil <\xi+1$ for every real $\xi$, we obtain
\begin{equation}\label{eq:baseline}
 b\equiv r\pmod Q,\qquad
 u+\frac L3\le b<u+\frac L3+Q.
\end{equation}
If $S=\varnothing$, we take $t=1$.

If $S\ne\varnothing$, then by \eqref{eq:shift-count} and
\eqref{eq:shift-size}, we may apply Lemma~\ref{lem:translate}
to $S$ and this choice of $b$, for sufficiently large $X$.
Let $1\le t\le e^x$ be the integer given by the lemma.

In either case, set
\[
 N=b+Qt.
\]

For $1\le j\le H$, we have $N+j>b\ge u+L/3$ by
\eqref{eq:baseline}. On the other hand, \eqref{eq:room} gives
\[
 N+j\le b+Qe^x+H
 <u+\frac L3+Q+Qe^x+H
 <u+\frac{2L}3.
\]
We have proved
\begin{equation}\label{eq:block-location}
 u+\frac L3<N+j<u+\frac{2L}3<v
 \qquad(1\le j\le H).
\end{equation}
Thus the block lies in $[X/2,X]$, by \eqref{eq:empty-gap}, and
\eqref{eq:block-location} places it more than $L/3$ from each endpoint
of $(u,v)$.

Fix $j\in\{1,\ldots,H\}$ and $a\in\cB$.
By \eqref{eq:empty-gap}, either $a\le u$ or $a\ge v$.
If $a\ge v$, then \eqref{eq:block-location} gives
\[
 N+j-a\le N+j-v<0.
\]
This rules out a prime difference whenever $a\ge v$, including $a>X$.

It remains to consider $a\le u$. Since $\cB\subseteq\cZ_{\ge0}$
and $u\le X$, we have $a\in\cB\cap[0,X]$, and hence
$s=j-a$ belongs to $S_0$ by \eqref{eq:initial-shifts}.
If $s\in S_0\setminus S$, some prime $p\mid Q$ divides $r+s$.
The congruence $N\equiv r\pmod Q$ implies $p\mid N+s$, and
$p\le x<y$ by \eqref{eq:room}.
If $s\in S$, Lemma \ref{lem:translate}
supplies a prime $p\le y$ dividing the same integer $N+s$.
Thus in both cases there is a prime $p\le y$ with $p\mid N+j-a$.

On the other hand, $a\le u$, \eqref{eq:block-location}, and
\eqref{eq:room} imply
\[
 N+j-a\ge N+j-u>\frac L3>y\ge p.
\]
Hence $N+j-a$ is composite. We conclude that $N+j\notin\PP+\cB$
for every $1\le j\le H$.

The lower bounds on $X$ used in the proof depend only on the
absolute constants in Lemma~\ref{lem:translate} and on estimates uniform
in $\cB$. They can therefore be satisfied by one absolute choice
of $X_*$. This completes the proof of Theorem \ref{thm:block}.

\end{proof}

\section{Additive complements and Romanoff numbers}\label{sec:complements}

\begin{proof}[Proof of Theorem~\ref{cor:complement}]
Let $\cB$ be an asymptotic additive complement, and choose $n_0\ge1$
such that every integer $n\ge n_0$ belongs to $\PP+\cB$.
With $c_*$ and $X_*$ as in Theorem~\ref{thm:block}, take
\[
 X\ge\max\{X_*,2n_0\}.
\]
Suppose, for a contradiction, that
\[
 B(X)+1\le c_*\log X\log\log X.
\]
Then \eqref{eq:H} gives $H\ge1$, and Theorem~\ref{thm:block}
supplies an integer
\[
 n\in[X/2,X]\setminus(\PP+\cB).
\]
But $n\ge X/2\ge n_0$, contradicting the choice of $n_0$.
It follows that
\begin{equation}\label{eq:complement-strict}
 B(X)+1>c_*\log X\log\log X
 \qquad\bigl(X\ge\max\{X_*,2n_0\}\bigr).
\end{equation}

For $X$ large enough that $c_*\log X\log\log X\ge2$,
\eqref{eq:complement-strict} yields
\[
 \begin{aligned}
 B(X)&>c_*\log X\log\log X-1\\
 &\ge\frac{c_*}{2}\log X\log\log X.
 \end{aligned}
\]
This proves \eqref{eq:lower} with the absolute constant $c=c_*/2$.
Dividing by $\log X>0$ gives
\[
 \frac{B(X)}{\log X}\ge c\log\log X\longrightarrow\infty,
\]
which is \eqref{eq:ruzsa}.
\end{proof}

\begin{proof}[Proof of Theorem~\ref{cor:base}]
Fix an integer $a\ge2$ and put
\[
 \cB=\{a^n:n\in\cN\}.
\]
Then $\PP+\cB=\cR_a$. For all sufficiently large $X$, we have
\[
 B(X)+1
 =\left\lfloor\frac{\log X}{\log a}\right\rfloor+1
 \le\frac{2\log X}{\log a}.
\]
Consequently, the length in Theorem~\ref{thm:block} satisfies
\[
 H=\left\lfloor c_*\frac{\log X\log\log X}{B(X)+1}\right\rfloor
 \ge\left\lfloor\frac{c_*\log a}{2}\log\log X\right\rfloor.
\]
For sufficiently large $X$, this length is positive, so
Theorem~\ref{thm:block} supplies a block of $H$ consecutive integers
in $[X/2,X]$ outside $\cR_a$. Hence
\[
 G_{\cR_a}(X)\ge H\gg_a\log\log X.
\]
This proves the theorem.
\end{proof}

\section*{Acknowledgements}
This work is an output of a research project (HSE-BR-2025-84) implemented as part of the Basic Research Program at HSE University.

\end{document}